\documentclass{mylms}

\usepackage{amsmath,amssymb,url}
\usepackage{rotating}
\newcommand{\Z}{{\mathbb{Z}}}
\newcommand{\Q}{{\mathbb{Q}}}
\newcommand{\N}{{\mathbb{N}}}
\newcommand{\R}{{\mathbb{R}}}
\newcommand{\C}{{\mathbb{C}}}
\newcommand{\cD}{{\mathcal{D}}}
\newcommand{\cH}{{\mathcal{H}}}
\newcommand{\cL}{{\mathcal{L}}}
\newcommand{\cS}{{\mathcal{S}}}
\newcommand{\ba}{{\mathbf{a}}}

\newcommand{\fC}{{\mathfrak{C}}}
\newcommand{\fX}{{\mathfrak{X}}}

\newcommand{\Irr}{{\operatorname{Irr}}}
\renewcommand{\leq}{\leqslant}
\renewcommand{\geq}{\geqslant}
\renewcommand{\atop}[2]{\genfrac{}{}{0pt}{}{#1}{#2}}

\newtheorem{thm}{Theorem}[section]
\newtheorem{lem}[thm]{Lemma}
\newtheorem{cor}[thm]{Corollary}
\newtheorem{prop}[thm]{Proposition}
\newtheorem{conj}[thm]{Conjecture}

\newnumbered{exmp}[thm]{Example}
\newnumbered{defn}[thm]{Definition}
\newnumbered{rem}[thm]{Remark}
\newnumbered{algo}{Algorithm}

\title[{\sf PyCox}]{{\sf PyCox}: Computing with (finite) Coxeter groups and 
Iwahori--Hecke algebras}

\author{Meinolf Geck}

\dedication{Dedicated to the memory of Professor H. Pahlings}

\classno{Primary 20C40, Secondary 20C08, 20F55}

\begin{document}
\maketitle

\begin{abstract} We introduce the computer algebra package {\sf PyCox}, 
written entirely in the {\sf Python} language. It implements a set of 
algorithms -- in a spirit similar to the older {\sf CHEVIE} system -- for 
working with Coxeter groups and Hecke algebras. This includes a new 
variation of the traditional algorithm for computing Kazhdan--Lusztig cells 
and $W$-graphs, which works efficiently for all finite groups of rank 
$\leq 8$ (except $E_8$). We also discuss the computation of Lusztig's 
leading coefficients of character values and distinguished involutions 
(which works for $E_8$ as well). Our experiments suggest a re-definition 
of Lusztig's ``special'' representations which, conjecturally, should also 
apply to the unequal parameter case. 
\end{abstract}

\section{Introduction} \label{sec0}
The computer algebra system {\sf CHEVIE} \cite{chv} has been designed to 
facilitate computations with various combinatorial structures arising 
in Lie theory, like finite Coxeter groups and Hecke algebras. It was
initiated about 20 years ago and has been further developed ever since;
see \cite{mycarp} for a discussion of some recent applications of this
system. However, there are some limitations to its use due to its 
dependence on {\sf GAP3} \cite{gap} which is still available, but no longer 
supported (the last release of {\sf GAP3} was in 1997). Therefore, it 
seemed desirable to implement a core set of algorithms around Coxeter 
groups and Hecke algebras in a more modern and widely available environment.
The success of {\sf Sage} \cite{sage} suggested the use of the 
{\sf Python} language \cite{python}. This lead to the development of
{\sf PyCox}, which we present in this note. 

Although in some areas of algebraic manipulations (like permutations, 
algebraic numbers) the performance is inferior to that of {\sf GAP3}, some 
of the advantages of the new system are: it works on every computer where
{\sf Python} is installed; we can now run jobs which use main memory well 
over 4GB; and we can just import {\sf PyCox} as a module into {\sf Sage}, 
whereby we have immediate access to all the further functionality of
{\sf Sage} (including the {\sf Sage} notebook and the interfaces to 
{\sf GAP4} and, even, {\sf GAP3}). 

In Section~\ref{design}, we briefly describe the basic design features 
of {\sf PyCox} and give some examples of how to use it; more details are 
available through the online help within {\sf PyCox}. 

We shall then discuss some concrete applications of our programs to 
questions related to the theory of Kazhdan--Lusztig cells and the 
associated $W$-graphs. (The basic definitions will be recalled in 
Section~\ref{secwgr}.) The problem of computing such cells has beed 
addressed by several authors, most notably by 
Alvis \cite{Al} and DuCloux \cite{state}, \cite{fokko}, \cite{Fokko}. In 
Section~\ref{secrel}, we present a variation of the known algorithm where 
the new ingredient is the use of ``relative'' Kazhdan--Lusztig polynomials, 
as defined in \cite{myind}, \cite{HowYi}. (Neither {\sf CHEVIE} nor 
DuCloux's {\sf Coxeter} \cite{fokko} contains an implementation of these.) 
As may be expected, the systematic use of these relative polynomials 
instead of the ordinary Kazhdan--Lusztig polynomials leads to a significant 
efficiency gain for the determination of left cells. (This idea was 
essentially already formulated by DuCloux \cite[\S 5.3]{state}; Howlett 
and Yin \cite{HowYi3} constructed $W$-graphs affording irreducible 
representations in this way.) While Howlett and Yin \cite{HowYi} only 
considered the equal parameter case, we shall describe recursion formulae 
in the general case of unequal parameters. 

Within {\sf PyCox}, it is now possible to deal with left cells and the
corresponding $W$-graphs for finite Coxeter groups of rank up to around $8$,
including types $H_4$, $E_6$ and $E_7$ but~--~not surprisingly~--~with the 
exception of type $E_8$. There are also programs for computing Lusztig's 
leading coefficients of character values \cite{Lu4} and distinguished 
involutions \cite{Lu2}; these do work for type $E_8$ as well. As far as 
we are aware, this is the first general program which is capable of 
dealing with this level of information for a group of size like that of 
type $E_7$ or $E_8$. As an example of an application we just mention that 
it is now straightforward to verify that Kottwitz's conjecture \cite{kott} 
on the intersections of left cells with conjugacy classes of involutions 
holds for type $E_7$, following the general methods explained by Casselman
\cite{cass}. (Abbie Halls, at the University of Aberdeen, is currently
working on type $E_8$, where more specialised methods and programming are 
required.) We also stress the fact that our programs only use results 
concerning cells which are generally known to hold by elementary 
arguments; in particular, we do not rely on any ``positivity'' results 
or facts concerning Lusztig's $\ba$-function \cite{Lu1}. 

Finally, in Section~\ref{secJ}, we consider the problem of computing
the character tables of certain symmetric algebras which have been
associated by Lusztig  \cite{Lu4} with the various left cells of a finite 
Coxeter group. In {\sf PyCox}, these tables can be determined by an 
automatic procedure for all groups of rank up to around $7$. In Lusztig's
work \cite{Lu4}, the non-crystallographic types $I_2(m)$, $H_3$ and $H_4$
have been excluded from the discussion. Here, we complete the picture by 
treating theses cases as well. This will allow us to formulate in full 
generality an alternative characterisation of ``special'' representations, 
which were originally defined by Lusztig \cite{Lusztig79b}. This 
characterisation should also make sense in the general case of unequal 
parameters; see Conjecture~\ref{cspec}. We discuss some examples 
to support this conjecture.

\section{Design of {\sf PyCox}} \label{design}
The whole {\sf PyCox} system is contained in one file accompanying this 
article; it is freely available for download, under the GPL licence. (The 
file is called {\tt chv1r6.py}, it has 13368 lines and the size is
roughly 522 KB; updates will also be made available at the author's 
homepage.)
So, in order to use {\sf PyCox} on your computer, all you need to do (once 
you have dowloaded the file) is to launch {\sf Python} (2.6 or higher) and 
import the file as a module, e.g., by typing:
\begin{verbatim}
    >>> from chv1r6 import *
\end{verbatim}
(You should then see a welcome banner.) Similarly, if {\sf Sage} (version 
4.7 or higher) is installed on your computer, you can import {\sf PyCox} 
as a module into {\sf Sage}:
\begin{verbatim}
    sage: from chv1r6 import *
\end{verbatim}
A good place to start is to type {\tt 'help(coxeter)'} or 
{\tt 'allfunctions()'}. 

We shall now discuss some of the basic principles of the system and show
some concrete examples. As in {\sf CHEVIE}, the basic object from which 
everything is built up is that of a Cartan matrix. Let $S$ be a finite 
non-empty index set and $C=(c_{st})_{s,t \in S}$ be a matrix with 
entries in~$\R$. Following \cite[\S 1.1]{gepf}, we say that $C$ is a 
{\em Cartan matrix} if the following conditions are satisfied:
\begin{enumerate}
\item[(C1)] For $s \neq t$ we have $c_{st} \leq 0$; furthermore, $c_{st}
\neq 0$ if and only if $c_{ts} \neq 0$.
\item[(C2)] We have $c_{ss}=2$ and, for $s \neq t$, we have $c_{st}
c_{ts}=4\cos^2(\pi/m_{st})$, where $m_{st} \geq 2$ is an integer or
$m_{st}=\infty$.
\end{enumerate}
Let $C=(c_{st})_{s,t \in S}$ be a Cartan matrix and $V$ be an ${\R}$-vector 
space of dimension $|S|$, with a fixed basis $\{\alpha_s \mid s \in S\}$. 
For each $s\in S$, we define a linear map on $V$ as follows: 
\[\dot{s} \colon V \rightarrow V, \quad \alpha_t \mapsto \alpha_t-c_{st} 
\alpha_s \quad (t \in S).\]
We shall act from the right so we use the row convention for matrices of 
linear maps. Let 
\[ W:=\langle \dot{s} \mid s \in S \rangle \subseteq \mbox{GL}(V).\]
In what follows, we shall often omit the dot when referring to the maps
$\dot{s} \colon V \rightarrow V$; in particular, $S$ will be regarded 
directly as a subset of $\mbox{GL}(V)$. With this convention, the group $W$
has a presentation with generators $S$ and defining relations as follows
(see \cite[1.2.7]{gepf}):
\begin{center}
$s^2=1$ for all $s \in S$ and $(st)^{m_{st}}=1$ for all $s\neq t$ in
$S$ with $m_{st}<\infty$.
\end{center}
Thus, $W$ is a Coxeter group and all Coxeter groups arise in this way.
The matrix $M=(m_{st})_{s,t \in S}$ is called the {\em Coxeter matrix}
of $W$. 

The set $\Phi:=\{ \alpha_s.w \mid s \in S, w \in W\}$ is the 
corresponding root system. There is a well-defined partition $\Phi=
\Phi^+ \amalg \Phi^-$ where $\Phi^+$ is the set of all $\alpha \in \Phi$ 
which can be expressed in terms of the basis $\{\alpha_s\mid s \in S\}$ 
where all coefficients are non-negative, and $\Phi^-=\{-\alpha\mid\alpha 
\in \Phi\}$; see \cite[1.1.9]{gepf}. Based on this information alone, we 
already have an efficient way of testing if an element $w \in W$ (given 
as a word in the generators in $S$) equals the identity or not: it
suffices to compute the corresponding linear map of $V$ and check if its
matrix is the identity or not. More generally, if $w \neq 1$, we can 
efficiently find an $s \in S$ such that $l(sw)<l(w)$ (see 
\cite[1.1.9]{gepf}):
\[ l(sw)=l(w)-1 \quad \mbox{if and only if} \quad \alpha_s.w\in \Phi^-.\]
(Here, $l(w)$ denotes the usual length of $w \in W$.)

Following the general ideas in {\sf CHEVIE}, the basic function in 
{\sf PyCox} is that of creating a Coxeter group from a Cartan matrix (as 
a {\sf Python} ``class''): 
\begin{verbatim}
    >>> W=coxeter([[2, -1, -1], [-1, 2, -1], [-1, -1, 2]])
\end{verbatim}
When the function {\tt coxeter} is called, it computes some basic 
information directly from the Cartan matrix, for example, the matrices of 
the generating reflections and the Coxeter matrix. These pieces of 
information are saved as components in the resulting {\sf Python} class:
\begin{verbatim}
    >>> W.coxetermat
    [[1, 3, 3], [3, 1, 3], [3, 3, 1]]
    >>> W.cartantype
    [['U', [0, 1, 2]]]       # 'U' stands for infinite 
    >>> W.matgens            # matrices of the generators
    [((-1, 0, 0), (1, 1, 0), (1, 0, 1)), 
     ((1, 1, 0), (0, -1, 0), (0, 1, 1)),
     ((1, 0, 1), (0, 1, 1), (0, 0, -1))]
\end{verbatim}
We can now do some basic operations with the elements of $W$:
\begin{verbatim}
    >>> m=W.wordtomat([1, 0, 1, 2, 1, 0]); m
    ((-1, 0, -1), (-2, -2, -1), (4, 3, 3))
    >>> W.mattoword(m)       # lexicographically smallest 
    [0, 1, 0, 2, 1, 0]       # reduced expression of m
    >>> W.leftdescentsetmat(m)
    [0, 1]
    >>> W.rightdescentsetmat(m)
    [0]
\end{verbatim}
When $W$ is finite, the function {\tt coxeter} also computes, for 
example, the number of positive roots and the order of $W$:
\begin{verbatim}
    >>> A=[[   2,  0, -1, -ir5,  0],   # ir5=(1+sqrt(5))/2; 
           [   0,  2,  0,    0, -1],   # see the function 
           [  -1,  0,  2,    0,  0],   # zeta5
           [-ir5,  0,  0,    2,  0], 
           [   0, -3,  0,    0,  2]]
    >>> W=coxeter(A) 
    >>> W.N; W.order; W.degrees
    21                           # number of positive roots
    1440                         # order of the group
    [2, 2, 6, 6, 10]             # reflection degrees
\end{verbatim}
There are further components, like {\tt W.roots} (containing a list of 
all roots), {\tt W.permgens} (the permutation action of the generators 
on the roots) etc.; the list of all components is obtained by {\tt dir(W)} 
(a general {\sf Python} function). In order to check if $W$ is finite, 
{\tt coxeter} uses the known list of Coxeter matrices of irreducible 
finite Coxeter groups, encoded in terms of graphs (with a standard 
labelling of the vertices) as in Table~\ref{Mcoxgraphs}. 

\begin{table}[htbp]
\caption{Coxeter graphs of irreducible finite Coxeter groups}
\label{Mcoxgraphs}
\begin{center} 
\begin{picture}(325,160)
\put( 15, 25){$E_7$}
\put( 40, 25){\circle*{5}}
\put( 38, 30){$0$}
\put( 40, 25){\line(1,0){20}}
\put( 60, 25){\circle*{5}}
\put( 58, 30){$2$}
\put( 60, 25){\line(1,0){20}}
\put( 80, 25){\circle*{5}}
\put( 78, 30){$3$}
\put( 80, 25){\line(0,-1){20}}
\put( 80,  5){\circle*{5}}
\put( 85,  3){$1$}
\put( 80, 25){\line(1,0){20}}
\put(100, 25){\circle*{5}}
\put( 98, 30){$4$}
\put(100, 25){\line(1,0){20}}
\put(120, 25){\circle*{5}}
\put(118, 30){$5$}
\put(120, 25){\line(1,0){20}}
\put(140, 25){\circle*{5}}
\put(138, 30){$6$}

\put(175, 25){$E_8$}
\put(200, 25){\circle*{5}}
\put(198, 30){$0$}
\put(200, 25){\line(1,0){20}}
\put(220, 25){\circle*{5}}
\put(218, 30){$2$}
\put(220, 25){\line(1,0){20}}
\put(240, 25){\circle*{5}}
\put(238, 30){$3$}
\put(240, 25){\line(0,-1){20}}
\put(240,  5){\circle*{5}}
\put(245,  3){$1$}
\put(240, 25){\line(1,0){20}}
\put(260, 25){\circle*{5}}
\put(258, 30){$4$}
\put(260, 25){\line(1,0){20}}
\put(280, 25){\circle*{5}}
\put(278, 30){$5$}
\put(280, 25){\line(1,0){20}}
\put(300, 25){\circle*{5}}
\put(298, 30){$6$}
\put(300, 25){\line(1,0){20}}
\put(320, 25){\circle*{5}}
\put(318, 30){$7$}

\put(  5, 90){$I_2(m)$}
\put(  7, 79){$\;\scriptstyle{m \geq 3}$}
\put( 40, 85){\circle*{5}}
\put( 38, 90){$0$}
\put( 47, 88){$\scriptstyle{m}$}
\put( 40, 85){\line(1,0){20}}
\put( 60, 85){\circle*{5}}
\put( 58, 90){$1$}

\put( 95, 84){$F_4$}
\put(120, 85){\circle*{5}}
\put(118, 90){$0$}
\put(120, 85){\line(1,0){20}}
\put(140, 85){\circle*{5}}
\put(138, 90){$1$}
\put(140, 85){\line(1,0){20}}
\put(148, 88){$\scriptstyle{4}$}
\put(160, 85){\circle*{5}}
\put(158, 90){$2$}
\put(160, 85){\line(1,0){20}}
\put(180, 85){\circle*{5}}
\put(178, 90){$3$}

\put( 15, 55){$H_3$}
\put( 40, 55){\circle*{5}}
\put( 38, 60){$0$}
\put( 40, 55){\line(1,0){20}}
\put( 49, 58){$\scriptstyle{5}$}
\put( 60, 55){\circle*{5}}
\put( 58, 60){$1$}
\put( 60, 55){\line(1,0){20}}
\put( 80, 55){\circle*{5}}
\put( 78, 60){$2$}

\put(115, 55){$H_4$}
\put(140, 55){\circle*{5}}
\put(138, 60){$0$}
\put(140, 55){\line(1,0){20}}
\put(149, 58){$\scriptstyle{5}$}
\put(160, 55){\circle*{5}}
\put(158, 60){$1$}
\put(160, 55){\line(1,0){20}}
\put(180, 55){\circle*{5}}
\put(178, 60){$2$}
\put(180, 55){\line(1,0){20}}
\put(200, 55){\circle*{5}}
\put(198, 60){$3$}

\put(215, 75){$E_6$}
\put(240, 75){\circle*{5}}
\put(238, 80){$0$}
\put(240, 75){\line(1,0){20}}
\put(260, 75){\circle*{5}}
\put(258, 80){$2$}
\put(260, 75){\line(1,0){20}}
\put(280, 75){\circle*{5}}
\put(278, 80){$3$}
\put(280, 75){\line(0,-1){20}}
\put(280, 55){\circle*{5}}
\put(285, 53){$1$}
\put(280, 75){\line(1,0){20}}
\put(300, 75){\circle*{5}}
\put(298, 80){$4$}
\put(300, 75){\line(1,0){20}}
\put(320, 75){\circle*{5}}
\put(318, 80){$5$}

\put( 13,120){$B_n$}
\put( 12,109){$\scriptstyle{n \geq 2}$}
\put( 40,115){\circle*{5}}
\put( 38,120){$0$}
\put( 40,115){\line(1,0){20}}
\put( 48,118){$\scriptstyle{4}$}
\put( 60,115){\circle*{5}}
\put( 58,120){$1$}
\put( 60,115){\line(1,0){30}}
\put( 80,115){\circle*{5}}
\put( 78,120){$2$}
\put(100,115){\circle*{1}}
\put(110,115){\circle*{1}}
\put(120,115){\circle*{1}}
\put(130,115){\line(1,0){10}}
\put(140,115){\circle*{5}}
\put(133,120){$n{-}1$}

\put( 13,150){$A_{n-1}$}
\put( 16,139){$\scriptstyle{n \geq 2}$}
\put( 60,145){\circle*{5}}
\put( 58,150){$0$}
\put( 60,145){\line(1,0){30}}
\put( 80,145){\circle*{5}}
\put( 78,150){$1$}
\put(100,145){\circle*{1}}
\put(110,145){\circle*{1}}
\put(120,145){\circle*{1}}
\put(130,145){\line(1,0){10}}
\put(140,145){\circle*{5}}
\put(133,150){$n{-}2$}

\put(200,127){$D_n$}
\put(200,116){$\scriptstyle{n \geq 3}$}
\put(220,145){\circle*{5}}
\put(225,145){$1$}
\put(220,105){\circle*{5}}
\put(226,100){$0$}
\put(220,145){\line(1,-1){21}}
\put(220,105){\line(1,1){21}}
\put(240,125){\circle*{5}}
\put(238,130){$2$}
\put(240,125){\line(1,0){30}}
\put(260,125){\circle*{5}}
\put(258,130){$3$}
\put(280,125){\circle*{1}}
\put(290,125){\circle*{1}}
\put(300,125){\circle*{1}}
\put(310,125){\line(1,0){10}}
\put(320,125){\circle*{5}}
\put(313,130){$n{-}1$}
\end{picture}
\end{center}
\end{table}

Note that, in general, there may be several Cartan matrices which 
give rise to the same Coxeter matrix. In {\sf PyCox} (as in {\sf CHEVIE}) 
we have adopted the following conventions:
\begin{itemize}
\item If $m_{st}$ is odd, then $c_{st}=c_{ts}$. (This has the
consequence that the root system is reduced; see \cite[1.3.6]{gepf}.)
\item If $m_{st}$ is even, then $c_{st}=-1$  or $c_{ts}=-1$. 
\end{itemize}
For example, the following two Cartan matrices both correspond to the 
Coxeter matrix of type $B_3$:
\begin{verbatim}
    >>> cartanmat("B",3)
    [[2, -2, 0], [-1, 2, -1], [0, -1, 2]]
    >>> cartanmat("C",3)
    [[2, -1, 0], [-2, 2, -1], [0, -1, 2]]
\end{verbatim}
See the help for the function {\tt cartanmat} for a detailed description
of the resulting choices of Cartan matrices for the various finite types. 

When {\tt coxeter} is called, it decomposes the Cartan matrix into 
its indecomposable components and checks if the corresponding Coxeter 
graphs appear in the list in Table~\ref{Mcoxgraphs}. If this is so, it 
matches the Cartan matrices of the indecomposable components to those 
returned by the function {\tt cartanmat}. This information is kept in the 
component {\tt W.cartantype}. In the above example (where $W$ is defined
by a $5 \times 5$ Cartan matrix), we have:
\begin{verbatim}
    >>> W.cartantype
    [['H', [3, 0, 2]], ['G', [1, 4]]]
\end{verbatim}
This means that the submatrix of $A$ with rows and columns indexed by 
$3,0,2$ (in this order) is the standard Cartan matrix of type $H_3$, as 
returned by calling {\tt cartanmat('H',3)}; similarly, the submatrix of 
$A$ with rows and columns indexed by $1,4$ (in this order) is the standard 
Cartan matrix of type $G_2$, as returned by calling {\tt cartanmat('G',2)}. 
In particular, we see that our group $W$ is of type $H_3 \times G_2$. 

The type recognition procedure is particularly helpful when dealing with 
reflection sub\-groups:
\begin{verbatim}
    >>> W=coxeter("F",4)  # same as coxeter(cartanmat("F",4))
    >>> H=reflectionsubgroup(W,[1,2,6,47])
    >>> H.cartantype                     # subgroup generated 
    [['C', [0, 1, 2]], ['A', [3]]]       # by  reflections at     
    >>> H.cartan                         # roots no. 1,2,6,47
    [[2,-1, 0, 0], [-2, 2,-1, 0], [0,-1, 2, 0], [0, 0, 0, 2]]
\end{verbatim}
Here, $H$ will be a Coxeter group in its own right. The information 
about the embedding into $W$ is held in the component {\tt H.fusions}; 
every Coxeter group in {\sf PyCox} has such a {\tt fusion} component:
it will at least contain the embedding into itself; see the online help 
of {\tt reflectionsubgroup} for further details. In the above example,
we have:
\begin{verbatim}
    >>>  W.cartanname    # unique string identifying W
    'F4c0c1c2c3'
    >>> H.fusions['F4c0c1c2c3']
    {'parabolic': False, 'subJ': [1, 2, 3, 23]}
\end{verbatim}
Thus, $H$ is not a parabolic subgroup and the four simple reflections
of $H$ correspond to the reflections with roots indexed by $1,2,3,23$ in
{\tt W.roots}. (This design is different from that in {\sf CHEVIE}; it 
appears to be better suited to recursive algorithms involving various 
reflection subgroups.) 

Let us assume from now on that $W$ is finite. Then, in principle, every 
piece of information about $W$ is ultimately computable from the Cartan 
matrix of $W$. However, as in {\sf CHEVIE}, some very basic and 
frequently used pieces of information are explicitly stored within the 
system; this is particularly relevant for data which are accompanied by
some more or less natural labellings (like partitions of $n$ for the 
conjugacy classes and irreducible characters of groups of type $A_{n-1}$).
In {\sf PyCox}, we store explicitly the following pieces of information
(with the appropriate labellings where this applies):
\begin{itemize}
\item reflection degrees  (see the function {\tt degreesdata});
\item conjugacy classes (see {\tt conjclassdata});
\item character tables (see {\tt irrchardata} and 
{\tt heckeirrdata});
\item Schur elements (see {\tt schurelmdata}).
\end{itemize}
For classical types $A_n$, $B_n$, $D_n$, this is done in the form of
combinatorial algorithms; for the remaining exceptional types, explicit
tables with the relevant information are stored. Then, for example, when the 
function {\tt chartable(W)} is called, {\sf PyCox} will build the character 
table of $W$ from the explicitly stored data for the irreducible components 
of $W$. (Note that $W$ is a direct product of its irreducible components, 
and there is a standard procedure to build the character table of a direct 
product of finite groups from the character tables of the direct factors.)

Let us now give a concrete example of how to use these programs. We would 
like to program a function which returns the list of involutions in $W$, 
that is, all the elements $w \in W$ such that $w^2=1$. To start somewhere, 
we have a look at the list of all available functions in {\sf PyCox}; this 
is printed by calling {\tt allfunctions()}. There is a function
{\tt allelmsproperty} which takes as input a group $W$ and a function 
$f\colon W\rightarrow \{{\tt True}, {\tt False}\}$; it returns the list 
of all $w \in W$ (as reduced words) such that $f(w)={\tt True}$. This 
certainly fits our problem: we just need to define $f$ such that $f(w)=
{\tt True}$ if $w$ has order $1$ or $2$, and $f(w)={\tt False}$ otherwise. 
Thus, our first candidate for the desired function is:
\begin{verbatim}
    >>> def involutions1(W):
    ...   return allelmsproperty(W,lambda x:W.permorder(W.wordtoperm(x))<=2)
    >>> len(involutions1(coxeter("E",6)))
    892
\end{verbatim}
This works fine for groups of moderate size but, eventually, we would also 
like to apply this to big examples like groups of type $E_7$ and $E_8$; 
however, when we do this, we notice that a long time will pass before we 
see a result. This is because {\tt allelmsproperty} is an ``all-purpose'' 
function which runs through all elements of $W$, transforms every element 
into a permutation and checks if this has order $1$ or $2$. For type $E_8$ 
with its 696{,}729{,}600 elements this will simply take too long. We need 
to tailor our program more specifically to the problem that we are dealing
with. Now, the set of involutions is invariant under conjugation so it 
will be a union of the conjugacy classes of $W$. The function 
{\tt conjugacyclasses} does return some information about the conjugacy 
classes of $W$, including representatives of the classes and the sizes of the
classes. (This uses data stored within the system; see {\tt conjclassdata}.) 
So, alternatively to our first try above, we could just select the class 
representatives which are involutions and then compute the corresponding 
conjugacy classes. In {\sf Python}, this can in fact be done in one line:
\begin{verbatim}
    >>> def involutions(W):
    ...   return flatlist([conjugacyclass(W, W.wordtoperm(w)) 
    ...            for w in conjugacyclasses(W)['reps'] 
    ...              if W.wordtocoxelm(2*w)==tuple(W.rank)])
\end{verbatim}
This even works in type $E_8$ where it returns the list of $199952$
involutions in about $1$ minute. (In {\sf GAP3}, the analogous function
would be roughly twice as fast, thanks to the much more efficient
arithmetic for permutations.) 
It is known that involutions play a special role in the theory of 
Kazhdan--Lusztig cells; see Lusztig \cite{Lu2}, Kottwitz \cite{kott}, 
Lusztig--Vogan \cite{luvo11}. We shall come back to this in 
Section~\ref{secJ}.

\section{Cells and $W$-graphs} \label{secwgr}

Let $W$ be a Coxeter group, with generating set $S$. In this section,
we briefly recall some basic definitions concerning left cells and the 
corresponding $W$-graphs, as introduced by Kazhdan and Lusztig \cite{KaLu}, 
\cite{Lusztig83}. Roughly speaking, these concepts give rise to a partition 
\[ W=\fC_1 \amalg \fC_2 \amalg  \ldots \amalg \fC_r\]
and, for each piece $\fC_i$ in this partition, a $W$-module $[\fC_i]_1$ 
with a standard basis $\{b_x \mid x \in \fC_i\}$ where the action of a 
generator $s \in S$ is described by formulae of a particularly simple 
form (encoded in a ``$W$-graph'', see Definition~\ref{defwgr} and 
Remark~\ref{rwgr} below). To give more precise definitions, we need to 
fix some notation. We shall work in the general multi-parameter framework 
of Lusztig \cite{Lusztig83}, \cite{Lusztig03}, which introduces a weight 
function into the picture on which all the subsequent constructions depend.

Let $\Gamma$ be an abelian group (written additively). Let $\{p_s\mid 
s\in S\} \subseteq \Gamma$ be a collection of elements such that $p_s=p_t$ 
whenever $s,t \in S$ are conjugate in $W$. This gives rise to a weight 
function 
\[ L \colon W\rightarrow\Gamma\]
in the sense of Lusztig \cite{Lusztig03}; we have $L(w)=p_{s_1}+\ldots+ 
p_{s_k}$ where $w=s_1\cdots s_k$ ($s_i \in S$) is a reduced expression 
for $w \in W$. We shall assume that $\Gamma$ admits a total ordering 
$\leq$ which is compatible with the group structure, that is, whenever 
$g, g' \in \Gamma$ are such that $g \leq g'$, we have $g+h\leq g'+h$ for 
all $h \in \Gamma$. We assume throughout that 
\[ L(s) \geq 0 \qquad \mbox{for all $s \in S$}.\]
(The original ``equal parameter'' setting of \cite{KaLu} corresponds to 
the case where $\Gamma=\Z$ with its natural ordering and $p_s=1$ for all 
$s \in S$.)

Furthermore, let $R \subseteq \C$ be a subring and $A=R[\Gamma]$ be the 
free $R$-module with basis $\{\varepsilon^g \mid g\in \Gamma\}$. (The
basic constructions in this section are independent of the choice of $R$
and so we could just take $R=\Z$ here; the flexibility of choosing $R$ will
be useful once we consider representations of $W$.) There is a well-defined 
ring structure on $A$ such that $\varepsilon^g \varepsilon^{g'}=
\varepsilon^{g+ g'}$ for all $g,g' \in \Gamma$. We write $1=\varepsilon^0 
\in A$. Let $\cH$ be the generic Iwahori--Hecke algebra corresponding to 
$(W,S)$, with parameters $\{\varepsilon^{L(s)}\mid s\in S\}$. Thus, $\cH$ 
has an $A$-basis $\{\tilde{T}_w\mid w \in W\}$ and the multiplication is 
given by the rule 
\[\tilde{T}_s\,\tilde{T}_w=\left\{\begin{array}{cl} \tilde{T}_{sw} & 
\quad \mbox{if $l(sw)>l(w)$},\\ \tilde{T}_{sw}+(\varepsilon^{L(s)}-
\varepsilon^{-L(s)}) \tilde{T}_w & \quad \mbox{if $l(sw)<l(w)$};
\end{array} \right.\]
here, $l\colon W \rightarrow {\N}_0$ denotes the usual length function
on $W$ with respect to $S$. 

Let $\Gamma_{\geq 0}:=\{g\in \Gamma\mid g\geq 0\}$ and denote by 
$A_{\geq 0}$ (or $R[\Gamma]_{\geq 0}$) the set of all $R$-linear 
combinations of terms $\varepsilon^g$ where $g\geq 0$. The notations 
$A_{>0}$, $A_{\leq 0}$, $A_{<0}$ (or $R[\Gamma_{>0}]$, $R[\Gamma_{\leq 0}]$, 
$R[\Gamma_{<0}]$) have a similar meaning. 

Let $a \mapsto \bar{a}$ be the $R$-linear involution of $A[\Gamma]$ which 
takes $g$ to $g^{-1}$ for any $g \in \Gamma$. This extends to a ring 
involution $\cH \rightarrow \cH$, $h \mapsto \bar{h}$, where
\[ \overline{\sum_{w \in W} a_w\tilde{T}_w}=\sum_{w \in W} \bar{a}_w
\tilde{T}_{w^{-1}}^{-1} \qquad (a_w \in A \mbox{ for all $w \in W$}).\]
We then have a corresponding {\em Kazhdan--Lusztig basis} of $\cH$, which 
we denote by $\{C_w'\mid w\in W\}$ (as in \cite{Lusztig83}). 
The basis element $C_w'$ is uniquely determined by the conditions that
\[\overline{C}_w'=C_w'\qquad \mbox{and}\qquad C_w'=\sum_{y \in W} 
P^{*}_{y,w}\, \tilde{T}_y\]
where $P_{w,w}^*=1$ and $P^{*}_{y,w}\in {\Z}[\Gamma_{<0}]$ if $y \neq w$;
furthermore, we have $P_{y,w}^*=0$ unless $y\leq w$, where $\leq$ denotes 
the Bruhat--Chevalley order on $W$. For $w \in W$ and $s \in S$, we have 
\[ \renewcommand{\arraystretch}{1.2} \tilde{T}_sC_w'=\left\{\begin{array}{cl} 
C_{sw}' & \mbox{ if $L(s)=0$},\\ \varepsilon^{L(s)} C_w' & 
\mbox{ if $L(s)>0$, $sw<w$},\\ \displaystyle C_{sw}'
-\varepsilon^{-L(s)}C_w'+ \sum_{\atop{y \in W}{sy<y<w}} M_{y,w}^s\, C_y'& 
\mbox{ if $L(s)>0$, $sw>w$}, \end{array}\right.\]
where $M_{y,w}^s$ are certain elements of ${\Z}[\Gamma]$ such that 
$\bar{M}_{y,w}^s=M_{y,w}^s$. As explained in \cite[\S 3]{Lusztig83},
these elements are determined by the inductive condition 
\begin{equation*}
M_{y,w}^s-\varepsilon^{L(s)}P_{y,w}^*+\sum_{\atop{z \in W}{sz<z, y<z<w}}
P_{y,z}^* M_{z,w}^s \in {\Z}[\Gamma_{<0}]\tag{M1}
\end{equation*}
and by the symmetry condition
\begin{equation*}
\overline{M}_{y,w}^s=M_{y,w}^s.\tag{M2}
\end{equation*}
By applying the anti-involution $\cH \rightarrow \cH$, $\tilde{T}_w \mapsto
\tilde{T}_{w^{-1}}$, we also obtain ``right-handed'' versions of the above
formulae (see \cite[\S 6]{Lusztig83}).

\begin{rem} \label{eqp} We set $P_{y,w}=\varepsilon^{L(w)-L(y)}
P_{y,w}^*$. Then it is known that $P_{y,w} \in {\Z}[\Gamma_{\geq 0}]$; see
Lusztig \cite[Prop.~5.4]{Lusztig03}. Furthermore, we have 
\[ \varepsilon^{L(s)} M_{y,w}^s \in {\Z}[\Gamma_{>0}] \qquad
\mbox{where $sy<y<w<sw$ ($s \in S$)};\]
see \cite[Prop.~6.4]{Lusztig03}. Assume now that $\Gamma=\Z$ and $L(s)=1$ 
for all $s \in S$ (equal parameter case, as in \cite{KaLu}). Then $A$ is the 
ring of Laurent polynomials in the indeterminate $\varepsilon$. Let $y,w 
\in W$ be such that $y<w$. Let $s \in S$ be such that $sy<y<w<s$. Now 
$P_{y,w}^*$ is a polynomial in $\varepsilon^{-1}$. Consequently, $P_{y,w}$ 
is a polynomial in $\varepsilon$ of degree at most $l(w)-l(y)-1$. In this 
situation, it is known that $M_{y,w}^s$ has the following simple description:
\begin{align*} 
M_{y,w}^s&=\mbox{ coefficient of $\varepsilon^{-1}$ in $P_{y,w}^*$}\\
&=\mbox{ coefficient of $\varepsilon^{l(w)-l(y)-1}$ in $P_{y,w}$};
\end{align*}
see Lusztig \cite[Cor.~6.5]{Lusztig03}.
\end{rem}

\begin{exmp} \label{rcrit} The determination of $M_{y,w}^s$ in the case 
of unequal parameters is considerably more involved than in the case of 
equal parameters. For example, assume that there exists some $t \in S$ 
such that $L(t)>0$, $ty>y$ and $tw<w$. Then $P_{y,w}^*=\varepsilon^{-L(t)} 
P_{ty,w}^*$. In the equal parameter case, this implies that $M_{y,w}^s=0$ 
unless $ty=w$, in which case $M_{y,w}^s=1$. In the general case of 
unequal parameters, if $ty=w$, we have
\[ M_{y,w}^s=\left\{\begin{array}{cl} 0 & \quad \mbox{if $L(s)<L(t)$},\\
1 & \quad \mbox{if $L(s)=L(t)$},\\
\varepsilon^{L(s)-L(t)}+\varepsilon^{L(t)-L(s)} & \quad \mbox{if
$L(s)>L(t)$};\end{array}\right.\]
see \cite[Prop.~5]{Lusztig83}. Furthermore, if $w\neq ty$, it can happen 
that $M_{y,w}^s\neq 0$.
\end{exmp}

\begin{defn}[(Kazhdan--Lusztig \protect{\cite{KaLu}} (equal 
parameter case); see \protect{\cite[1.4.11]{geja}} for general $L$)] 
\label{defwgr} A {\em $W$-graph} for $\cH$ consists of the following data:
\begin{itemize}
\item[(a)] a base set $\fX$ together with a map $I$ which assigns to
each $x \in \fX$ a subset $I(x) \subseteq S$;
\item[(b)] for each $s\in S$ with $L(s)>0$, a collection of 
elements 
\[ \{m_{x,y}^s\mid x,y \in \fX \mbox{ such that } s \in I(x),\,
s \not\in I(y)\};\]
\item[(c)] for each $s\in S$ with $L(s)=0$, a bijection
$\fX \rightarrow \fX$, $x \mapsto s.x$. 
\end{itemize}
These data are subject to the following requirements. First we require
that, for any $x,y \in \fX$ and $s \in S$ where $m_{x,y}^s$ is defined, we 
have 
\[\varepsilon^{L(s)}m_{x,y}^s\in R[\Gamma_{>0}]\qquad\mbox{and}
\qquad\overline{m}_{x, y}^s=m_{x,y}^s.\]
Furthermore, let $[\fX]$ be a free $A$-module with a basis $\{b_y \mid
y\in \fX\}$. For $s \in S$, define an $A$-linear map $\rho_s \colon [\fX]
\rightarrow [\fX]$ by
\begin{equation*}
\renewcommand{\arraystretch}{1.2}
\rho_s(b_y)= \left\{\begin{array}{cl} 
b_{s.y} & \qquad\mbox{if $L(s)=0$},\\
-\varepsilon^{-L(s)}\,b_y &\qquad \mbox{if $L(s)>0$, $s \in I(y)$}, \\
\displaystyle{\varepsilon^{L(s)}\,b_y+\sum_{x \in \fX:\,s \in I(x)} 
m_{x,y}^s \,b_{x}} &\qquad\mbox{if $L(s)>0$, $s\not\in I(y)$}.
\end{array}\right.
\end{equation*}
Then we require that the assignment $\tilde{T}_s \mapsto \rho_s$ 
defines a representation of $\cH$.
\end{defn}

\begin{exmp}[(Kazhdan--Lusztig \cite{KaLu}, Lusztig \cite{Lusztig83})]
\label{explcell} Let $y,z\in W$. We write $z \leftarrow_{\cL} 
y$ if there exists some $s \in S$ such that $C_z'$ appears with non-zero
multiplicity in $C_s'C_y'$ (when expressed in the $C'$-basis of $\cH$).
Thus, we have:
\[ z \leftarrow_{\cL} y \quad \Leftrightarrow \quad \left\{\begin{array}{l}
\mbox{if $z=sy$ for some $s \in S$, where $L(s)=0$ or $sy>y$},\\
\mbox{or if $M_{z,y}^s \neq 0$ for some $s \in S$, where 
$L(s)>0$ and $sz<z<y<sy$}.  \end{array}\right.\]
Let $\leq_{\cL}$ be the pre-order relation on $W$ generated by 
$\leftarrow_{\cL}$, that is, we have $z \leq_{\cL} y$ if there exist 
elements $z=y_0, y_1,\ldots,y_m=y$ in $W$ such that $y_{i-1} 
\leftarrow_{\cL} y_i$ for $1\leq i\leq m$. Let $\sim_{\cL}$ denote the 
associated equivalence relation; the corresponding equivalence classes 
are called the {\em left cells} of $W$.

Let $\fC$ be a left cell of $W$ (or, more generally, a union of left cells).
Then we obtain a corresponding $W$-graph as follows. We set $I(x):=\{ s 
\in S \mid sx<x\}$ for $x \in \fC$. Furthermore, if $x,y\in \fC$ and 
$s \in S$ are such that $L(s)>0$, $s \in I(x)$ and $s \not\in I(y)$, we set
\[ m_{x,y}^s:=\left\{\begin{array}{cl} 1 & \qquad \mbox{if $y=sx$},\\
-(-1)^{l(x)+l(y)}M_{x,y}^s & \qquad \mbox{if $x<y$},\\ 0 & \qquad
\mbox{otherwise}.\end{array}\right.\]
Finally, if $s \in S$ is such that $L(s)=0$, then $sw \in \fC$ for all 
$w \in \fC$, so we obtain a natural bijection $\fC \rightarrow \fC$ by left 
multiplication. It is known that these data give rise to a $W$-graph 
structure on the set $\fC$. (See \cite[\S 6]{Lusztig83}.)
\end{exmp}

\begin{rem} \label{rwgr} Let $\theta \colon A \rightarrow R$ be the unique 
$R$-linear ring homomorphism such that $\theta(\varepsilon^g)=
1$ for all $g \in \Gamma$. Then, regarding $R$ as an $A$-module via $\theta$,
we have $R \otimes_A \cH \cong R[W]$, the group algebra of $W$ over $R$.
Let $\fC$ be a left cell of $W$. Then we obtain a representation of $W$
on $[\fC]_1:=R \otimes_A [\fC]$, called a ``left cell representation'' of 
$W$. If $W$ is a finite Weyl group and $R=\Q$, the study of these 
left cell representations is of considerable interest in the representation
theory of reductive algebraic groups over finite fields; see Lusztig 
\cite{LuBook}.
\end{rem}

\begin{defn} \label{eqcell} Assume we are given two $W$-graphs with 
underlying base sets $\fX$ and $\fX'$. Then we say that these two 
$W$-graphs are {\em equivalent} if there exists a bijection 
$\fX \rightarrow \fX'$, $x \mapsto x'$, such that the map 
\[ [\fX]\rightarrow [\fX'], \qquad b_x \mapsto b_{x'},\]
is an $\cH$-module isomorphism. Similarly, if $\fC,\fC'$ are left cells
of $W$, we write $\fC \approx \fC'$ if the $W$-graphs associated with
$\fC$ and $\fC'$ are equivalent.
\end{defn}

\begin{exmp}[(Kazhdan--Lusztig \protect{\cite[\S 4]{KaLu}})] \label{star} 
Assume that we are in the equal parameter case where $\Gamma=\Z$ and $L(s)=1$
for all $s \in S$. Let $s,t \in S$ be such that $st$ has order~$3$. Let 
\[ D_R(s,t)=\{w \in W \mid \mbox{either } ws<w,wt>w \mbox{ or } 
ws>w, wt<w \}.\]
If $w \in D_R(s,t)$, then exactly one of the two elements $ws,wt$
belongs to $D_R(s,t)$; we denote it $w^*$. Thus, we obtain an involution 
\[ D_R(s,t) \rightarrow D_R(s,t), \qquad w \mapsto w^*.\]
If $\fC$ is a left cell of $W$, then it is known that either 
$\fC$ is contained in $D_R(s,t)$ or does not meet $D_R(s,t)$ at all;
see \cite[Prop.~2.4]{KaLu}. This also shows that $y^{-1}w \not\in \langle 
s,t\rangle$ for all $y\neq w$ in $\fC$. Now, if $\fC \subseteq D_R(s,t)$,
then 
\[ \fC^*=\{w^* \mid w \in \fC\}\subseteq D_R(s,t)\]
also is a left cell of $W$ (see \cite[Cor.~4.4(ii)]{KaLu}); furthermore, 
the $W$-graphs corresponding to $\fC$ and $\fC^*$ yield identical matrix 
representations of $\cH$ (see \cite[Theorem~4.2(iii)]{KaLu}). Thus, we 
have $\fC \approx \fC^*$ in the sense of Definition~\ref{eqcell}, where
the bijection is given by $w \mapsto w^*$ ($w \in \fC$).
\end{exmp}

\begin{defn}[(Cf.\ Lusztig \protect{\cite{Lu2}, \cite[14.2]{Lusztig03}})]
\label{disti} Let $w \in W$ and assume that $P_{1,w}^*\neq 0$. We define 
an element $\Delta(w)\in \Gamma_{\geq 0}$ and an integer $0\neq n_w\in 
\Z$ by the condition $\varepsilon^{\Delta(w)}\,P_{1,w}^* \equiv n_w
\bmod {\Z}[\Gamma_{<0}]$. Then we say that $w$ is {\em distinguished} (with
respect to $L$) if $\Delta(w)<\Delta(y)$ for any $y\neq w$ such that 
$P_{1,y}^*\neq 0$ and $y,w$ belong to the same left cell of $W$. We set 
\[ \cD:=\{w \in W\mid \mbox{ $w$ distinguished}\}.\]
Thus, if $w \in \cD$ and $\fC$ is the left cell containing~$w$, then the
function
\[ \{y\in \fC \mid P_{1,y}^*\neq 0\} \rightarrow \Gamma, \qquad y \mapsto
\Delta(y),\]
reaches its minimum at $w$ and $w$ is uniquely determined by this property.
(It is known that every left cell contains at least one element $y$ such 
that $P_{1,y}^* \neq 0$; see, for example, \cite[2.4.7]{geja}.)
\end{defn}

In the equal parameter case where $\Gamma=\Z$ and $L(s)=1$ for all $s\in S$
(and assuming that $W$ is finite) it is known that $w^2=1$ and $n_w=1$ for 
all $w \in \cD$; furthermore, every left cell contains a (unique) 
distinguished element. (See Lusztig \cite{Lu2}; see \cite{Fokko} for $W$ 
of non-crystallograhic type.) Hence, in particular, $\cD$ is a canonical 
set of representatives for the left cells of $W$. If $W$ is of type $A$, 
then $\cD$ consists precisely of all involutions in $W$; in general, $\cD$ 
is strictly contained in the set of involutions of $W$. 

We shall now be interested in determining the above data explicitly,
especially for groups of exceptional type. Thus, the computational 
tasks are:
\begin{itemize}
\item Given $W,L$, determine the partition of $W$ into left cells;
\item for each left cell $\fC$, determine the numbers 
$\{M_{x,y}^s\}$;
\item determine the set $\cD$ of distinguished elements (or the 
related set $\tilde{\cD}$ in Conjecture~\ref{conj42} below).
\end{itemize}
The crucial ingredient in these tasks is the computation of the
polynomials $P_{y,w}^*$. This is usually done using some known recursion 
formulae. In the next section, we discuss a variation of this recursion.

\section{Relative Kazhdan--Lusztig polynomials} \label{secrel}

We keep the general setting of the previous section. In addition, we 
shall now fix a subset $S'\subseteq S$ and consider the corresponding
standard parabolic subgroup $W'=\langle S'\rangle$. Let $X\subseteq W$ be
the set of distinguished left coset representatives of $W'$ in $W$. Every
element $w\in W$ can be written uniquely in the form $w=xu$ where $x\in X$,
$u\in W$ and $l(w)=l(x)+l(u)$; see \cite[\S 2.1]{gepf}. We shall frequently
use the following fact, due to Deodhar (see \cite[2.1.2]{gepf}). 
Let $x \in X$ and $s \in S$. Then we are in exactly one of the following
three cases:
\begin{enumerate}
\item $sx<x$ and $sx \in X$;
\item $sx>x$ and $sx \in X$;
\item $sx>x$ and $sx \not\in X$, in which case $sx=xt$ where 
$t \in S'$.
\end{enumerate}
We have a corresponding parabolic subalgebra $\cH'= \langle \tilde{T}_w
\mid w\in W'\rangle_A\subseteq \cH$. It is known that, for $w \in W'$, the 
basis element $C_w'$ lies in $\cH'$, and it is the Kazhdan--Lusztig basis 
element in $\cH'$. 

Let $y \in X$ and $v\in W'$. By \cite[Prop.~3.3]{myind}, we have a 
unique expression 
\[ C_{yv}'=\sum_{x \in X,u \in W'} p_{xu,yv}^* \tilde{T}_xC_u'\]
where $p_{yv,yv}^*=1$ and $p_{xu,yv}^*\in A_{<0}$ if $xu\neq yv$; 
furthermore, $p_{xu,yv}^*=0$ unless $xu=yv$ or $x<y$. In the proof of 
\cite[Prop.~3.3]{myind}, we have also seen that 
\[ \renewcommand{\arraystretch}{1.5} P_{xu,yv}^*=\left\{\begin{array}{cl}
P_{u,v}^* & \quad \mbox{if $x=y$},\\ \displaystyle p_{xu,yv}^*+
\sum_{\atop{w \in W'}{u<w}} P_{u,w}^* p_{xw,yv}^*& \quad 
\mbox{if $x<y$}.\end{array}\right.\]
Thus, if we have an efficient algorithm for computing the polynomials
$p_{xu,yv}^*$, then we can also determine $P_{xu,yv}^*$ and, hence,
the elements $\{M_{xu,yv}^s\}$.

\begin{prop} \label{recursp} We have the following recursion formulae
for $p_{xu,yv}^*$. 
\begin{itemize}
\item[(a)] If $y=1$, then  
\[ \renewcommand{\arraystretch}{1.2} 
p_{xu,v}^*=\left\{\begin{array}{cl} 1 & \mbox{ if $x=1$ and $u=v$},
\\ 0 & \mbox{ otherwise}.\end{array}\right.\]
\item[(b)] Now assume that $y \neq 1$ and let $s \in S$ be such that 
$sy<y$. If $L(s)=0$, then
\[ \renewcommand{\arraystretch}{1.2} p_{xu,yv}^*=\left\{\begin{array}{cl} 
p_{sxu,syv}^* & \quad \mbox{if $sx \in X$},\\ p_{xtu,syv}^* & \quad 
\mbox{if $sx\not\in X$},  \end{array}\right.\]
where $t=x^{-1}sx\in S'$ (if $sx\not\in X$). If $L(s)>0$, then
\[\renewcommand{\arraystretch}{1.5} p_{xu,yv}^*=\left\{\begin{array}{cl} 
p_{sxu,syv}^*+\varepsilon^{L(s)}p_{xu,syv}^*-\tilde{p}_{xu,yv}^s & 
\mbox{ if $sx<x$},\\ \varepsilon^{-L(s)}p_{sxu,yv}^* & \mbox{ if $sx>x$, 
$sx \in X$},\\ 0 & \mbox{ if $sx \not\in X$, $tu>u$}, \\
\begin{array}{l} \displaystyle (\varepsilon^{L(s)}+\varepsilon^{-L(s)})
p_{xu,syv}^*-\tilde{p}_{xu,yv}^s \\ \displaystyle  \quad 
+p_{xtu,syv}^*+\sum_{\atop{w \in W'}{u<w<tw}} M_{u,w}^tp_{xw,syv}^*
\end{array}& \mbox{ if $sx \not\in X$, $tu<u$}, \end{array}\right.\]
where $t=x^{-1}sx\in S'$ (if $sx\not\in X$) and 
\[\tilde{p}_{xu,yv}^s:=\sum_{\atop{z \in X, w \in W'}{x\leq z\leq sz
\text{ and } szw<zw<syv}} p_{xu,zw}^*\, M_{zw,syv}^s.\]
\end{itemize}
\end{prop}

\begin{proof} (a) This is contained in \cite[Prop.~3.3]{myind}. 

(b) This is essentially the same as the proofs of \cite[Theorem~5.1]{HowYi}
and \cite[Prop.~4.1]{HowYi2}. However, because of the different
normalisations and conventions, we shall sketch the main steps. Let 
$y\neq 1$ and $s \in S$ be such that $sy<y$. First assume that $L(s)=0$. 
Then $C_s'=\tilde{T}_s$ and $C_s'C_{syv}'=C_{yv}'$.
Furthermore,
\begin{align*}
C_s'C_{syv}'&=\sum_{x \in X,u \in W'} p_{xu,syv}^* \tilde{T}_s\tilde{T}_x
C_u'\\
&=\sum_{\atop{x \in X,u \in W'}{sx\in X}} p_{xu,syv}^* \tilde{T}_{sx}C_u'+
\sum_{\atop{x \in X,u \in W'}{sx\not\in X,\,sx=xt}} p_{xu,syv}^* \tilde{T}_x
\tilde{T}_tC_u'\\ &=\sum_{\atop{x \in X,u \in W'}{sx\in X}} p_{xu,syv}^* 
\tilde{T}_{sx}C_u'+ \sum_{\atop{x \in X,u \in W'}{sx\not\in X,\,sx=xt}} 
p_{xu,syv}^* \tilde{T}_xC_{tu}',
\end{align*}
where the last equality holds since $L(t)=L(s)$. This yields the desired
formulae. 

From now on, assume that $L(s)>0$. We begin by considering the identity 
$\tilde{T}_sC_{yv}'= \varepsilon^{L(s)}C_{yv}'$. The coefficient of 
$\tilde{T}_xC_u'$ on the right hand side is $\varepsilon^{L(s)}
p_{xu,yv}^*$. Now we compute 
\begin{align*} 
\tilde{T}_sC_{yv}'&=  \sum_{x \in X,u \in W'} p_{xu,yv}^* \tilde{T}_s
\tilde{T}_xC_u'\\
&=\sum_{\atop{x \in X,u\in W'}{sx<x}} p_{xu,yv}^*\tilde{T}_{sx}C_u'+
\sum_{\atop{x \in X,u\in W'}{sx<x}} p_{xu,yv}^*(\varepsilon^{L(s)}-
\varepsilon^{-L(s)})\tilde{T}_xC_u'\\ 
& \qquad +\sum_{\atop{x \in X,u \in W'}{sx>x, sx \in X}} p_{xu,yv}^* 
\tilde{T}_{sx}C_u'+\sum_{\atop{x \in X,w\in W'}{sx=xt \text{ where } 
t\in S'}} p_{xw,yv}^*\tilde{T}_x(\tilde{T}_tC_w')\\ 
&=\sum_{\atop{x \in X,u\in W'}{sx>x,sx \in X}} p_{sxu,yv}^*\tilde{T}_xC_u'+
\sum_{\atop{x \in X,u\in W'}{sx<x}} p_{xu,yv}^*(\varepsilon^{L(s)}-
\varepsilon^{-L(s)})\tilde{T}_xC_u'\\ 
& \qquad +\sum_{\atop{x \in X,u \in W'}{sx<x}} p_{sxu,yv}^* \tilde{T}_xC_u'+
\sum_{\atop{x \in X,w\in W'}{sx=xt \text{ where } t\in S'}} 
p_{xw,yv}^*\tilde{T}_x(\tilde{T}_tC_w').
\end{align*}
Thus, if $sx>x$ and $sx \in X$, then the coefficient of $\tilde{T}_x
C_u'$ in this expression is $p_{sxu,yv}^*$. Hence, we obtain $p_{sxu,yv}^*=
\varepsilon^{L(s)}p_{xu,yv}^*$ in this case, as required. 

Now assume that $sx>x$ and $sx \not\in X$. Then, among the various sums
in the above expression for $\tilde{T}_sC_{yv}'$, the term $\tilde{T}_x
C_u'$ will only appear in the sum 
\[\sum_{\atop{x \in X,w\in W'}{sx=xt \text{ where } t\in S'}} 
p_{xw,yv}^*\tilde{T}_x(\tilde{T}_tC_w').\]
If $tw<w$, then $\tilde{T}_tC_w'=\varepsilon^{L(t)}C_w'$. On the other 
hand, if $tw>w$, then $\tilde{T}_tC_w'$ is equal to $-\varepsilon^{-L(t)}
C_w'$ plus an $A$-linear combination of terms $C_{w'}'$ where $tw'<w'$. 
Hence, if $tu>u$, then the coefficient of $\tilde{T}_xC_u'$ in 
$\tilde{T}_sC_{yv}'$ will be $-\varepsilon^{-L(t)}p_{xu,yv}^*$. 
Thus, we have $-\varepsilon^{-L(t)}p_{xu,yv}^*=\varepsilon^{L(s)}
p_{xu,yv}^*$. Since $L(s)=L(t)$ and $\varepsilon^{2L(s)}\neq -1$, we 
deduce that $p_{xu,yv}^*=0$, as required.

To obtain the remaining formulae, we now consider the identity
\[ \tilde{T}_sC_{syv}'=C_{yv}'-\varepsilon^{-L(s)}C_{syv}'+
\sum_{\atop{z \in X, w \in W'}{swz<wz<syv}} M_{zw,syv}^s C_{zw}'.\]
Writing $C_{zw}'=\sum_{x \in X, u \in W'} p_{xu,zw}^* \tilde{T}_xC_u'$,
we obtain that 
\begin{align*}
\sum_{\atop{z \in X, w \in W'}{swz<wz<syv}} M_{zw,syv}^s C_{zw}'&=
\sum_{x \in X, u \in W'} \Bigl(\sum_{\atop{z \in X, w \in W'}{swz<wz<syv}} 
p_{xu,zw}^*\,M_{zw,syv}^s\Bigr)\tilde{T}_xC_u'\\&=\sum_{x \in X, 
u \in W'} \tilde{p}_{xu,yv}^s\tilde{T}_xC_u'.
\end{align*}
Thus, we have
\begin{align*}
C_{yv}'&=\tilde{T}_sC_{syv}'+\varepsilon^{-L(s)}C_{syv}'-\sum_{x \in X, 
u \in W'} \tilde{p}_{xu,yv}^s\tilde{T}_xC_u'\\
&=\tilde{T}_sC_{syv}'+\sum_{x \in X,u \in W'}\Bigl(\varepsilon^{-L(s)}
p_{xu,syv}^* -\tilde{p}_{xu,yv}^s\Bigr)\tilde{T}_xC_u'.
\end{align*}
By a similar computation as above, we have
\begin{align*}
\tilde{T}_sC_{syv}'&=\sum_{\atop{x \in X,u\in W'}{sx>x,sx \in X}} 
p_{sxu,syv}^*\tilde{T}_xC_u'+ \sum_{\atop{x \in X,u\in W'}{sx<x}} 
p_{xu,syv}^*(\varepsilon^{L(s)}-\varepsilon^{-L(s)})\tilde{T}_xC_u'\\ 
& \qquad +\sum_{\atop{x \in X, u \in W'}{sx<x}} p_{sxu,syv}^* \tilde{T}_x
C_u'+ \sum_{\atop{x \in X, w\in W'}{sx=xt \text{ where } t\in S'}} 
p_{xw,syv}^*\tilde{T}_x (\tilde{T}_tC_w').
\end{align*}
Now let $x \in X$ be such that $sx<x$. Then we conclude that 
\begin{align*}
p_{xu,yv}^*&=p_{xu,syv}^*(\varepsilon^{L(s)}-\varepsilon^{-L(s)})+
p_{sxu,syv}^*+\varepsilon^{-L(s)}p_{xu,syv}^*
-\tilde{p}_{xu,yv}^s\\ &=\varepsilon^{L(s)}p_{xu,syv}^*+p_{sxu,syv}^* 
-\tilde{p}_{xu,yv}^s,
\end{align*}
as required. Finally, assume that $sx>x$, $sx \not\in X$ and $tu<u$. Then
$p_{xu,yv}^*$ will be equal to $\varepsilon^{-L(s)}p_{xu,syv}^*-
\tilde{p}_{xu,yv}^s$ plus the coefficient of $\tilde{T}_xC_u'$ in 
\[ \sum_{\atop{x \in X, w\in W'}{sx=xt \text{ where } t\in S'}} 
p_{xw,syv}^*\tilde{T}_x (\tilde{T}_tC_w').\]
If $tw<w$, then $\tilde{T}_tC_w'=\varepsilon^{L(t)}C_w'$. On the other 
hand, if $tw>w$, then 
\[ \tilde{T}_tC_w'=C_{tw}'-\varepsilon^{-L(t)}C_w'+
\sum_{\atop{u \in W'}{tu<u<w}} M_{u,w}^t C_u'.\]
Hence, since $L(s)=L(t)$, we obtain 
\[ p_{xu,yv}^*=(\varepsilon^{L(s)}+\varepsilon^{-L(s)})p_{xu,syv}^*-
\tilde{p}_{xu,yv}^s+p_{xtu,syv}^*+ \sum_{\atop{w \in W'}{u<w<tw}} 
M_{u,w}^tp_{xw,syv}^*,\]
as desired.
\end{proof}

\begin{rem} \label{oldcase} Assume that $W'=\{1\}$. Then $X=W$ and
$P_{x,y}^*=p_{x,y}^*$ for all $x,y \in X$. For any $x \in X$ and 
$s \in S$, we have $sx \in X$ and either $sx<x$ or $sx>x$. Thus, only the 
first two out of the four cases in Proposition~\ref{recursp}(b) will occur. 
These two cases yield the known recursion formulae for the polynomials 
$P_{x,y}^*$.
\end{rem}

\begin{rem} \label{renorm1} Let us set $p_{xu,yv}=\varepsilon^{L(yv)-L(xu)}
p_{xu,yv}^*$. Then $p_{xu,yv}\in A_{\geq 0}$. Indeed, we have the 
recursion formulae:
\begin{itemize}
\item[(a)] If $y=1$, then  
\[ \renewcommand{\arraystretch}{1.2} 
p_{xu,v}=\left\{\begin{array}{cl} 1 & \mbox{ if $x=1$ and $u=v$},
\\ 0 & \mbox{ otherwise}.\end{array}\right.\]
\item[(b)] Now assume that $y \neq 1$ and let $s \in S$ be such that 
$sy<y$. If $L(s)=0$, then
\[ \renewcommand{\arraystretch}{1.2} p_{xu,yv}=\left\{\begin{array}{cl} 
p_{sxu,syv} & \quad \mbox{if $sx \in X$},\\ p_{xtu,syv} & \quad 
\mbox{if $sx\not\in X$},  \end{array}\right.\]
where $t=x^{-1}sx\in S'$ (if $sx\not\in X$). If $L(s)>0$, then
\[ \renewcommand{\arraystretch}{1.5} p_{xu,yv} =\left\{
\begin{array}{c@{\hspace{2mm}}l} 
p_{sxu,syv}+\varepsilon^{2L(s)}p_{xu,syv}-\hat{p}_{xu,yv}^s & 
\mbox{if $sx<x$},\\ p_{sxu,yv} & \mbox{if $sx>x$, $sx \in X$},\\ 
0 & \mbox{if $sx \not\in X$, $tu>u$}, \\
\begin{array}{l} \displaystyle (\varepsilon^{2L(s)}{+}1)p_{xu,syv}-
\hat{p}_{xu,yv}^s +\varepsilon^{2L(s)}p_{xtu,syv} \\[2mm] 
\displaystyle  \quad + \sum_{\atop{w \in W'}{u<w<tw}} 
\varepsilon^{L(tw)-L(u)}M_{u,w}^t p_{xw,syv} 
\end{array}& \mbox{if $sx \not\in X$, $tu<u$}, \end{array}\right.\]
where $t=x^{-1}sx\in S'$ (if $sx\not\in X$) and 
\[\hat{p}_{xu,yv}^s=\sum_{\atop{z \in X, w \in W'}{x\leq z \leq sy
\text{ and } szw<zw<syv}} \varepsilon^{L(syv)-L(zw)}\,p_{xu,zw}\, 
(\varepsilon^{L(s)}M_{zw,syv}^s).\]
\end{itemize}
With this renormalisation, it also follows that 
\[ \renewcommand{\arraystretch}{1.5} P_{xu,yv}=\left\{\begin{array}{cl}
P_{u,v} & \quad \mbox{if $x=y$},\\ \displaystyle p_{xu,yv}+
\sum_{\atop{w \in W'}{u<w}} P_{u,w} p_{xw,yv}& \quad
\mbox{if $x<y$}.\end{array}\right.\]
\end{rem}
 
\begin{lem} \label{lemmue} Let $x,y \in X$, $u,v \in W'$ and $s \in S$ be 
such that $L(s)>0$ and $sxu<xu<yv<syv$. Let 
\[\pi_{xu,yv}^s:=\sum_{\atop{z \in X, w \in W'}{x<z\leq y \text{ and }
szw<zw<yv}} p_{xu,zw}^*\, M_{zw,yv}^s.\]
Then the following hold. If $sx<x$, then 
\begin{equation*}
M_{xu,yv}^s-\varepsilon^{L(s)}p_{xu,yv}^*+\pi_{xu,yv}^s \in A_{<0};
\tag{$\text{M1}^\prime$}
\end{equation*}
on the other hand, if $sx \not\in X$ and $tu<u$ (where $t=x^{-1}sx\in S'$), 
then 
\begin{equation*}
M_{xu,yv}^s-\varepsilon^{L(s)}p_{xu,yv}^*+\pi_{xu,yv}^* 
-\sum_{\atop{w \in W'}{u<w<tw}} M_{u,w}^t p_{xw,yv}^* \in A_{<0}.
\tag{$\text{M1}^{\prime\prime}$}
\end{equation*}
Note that these conditions, together with the symmetry condition
$\overline{M}_{xu,yv}^s=M_{xu,yv}^s$, determine $M_{xu,yv}^s$
inductively.
\end{lem}

\begin{proof} We consider the identity
\begin{equation*}
C_{syv}'=\varepsilon^{-L(s)}C_{yv}'+\tilde{T}_sC_{yv}'-\sum_{\atop{z\in X,
w \in W'}{swz<wz<yv}} M_{zw,yv}^s C_{zw}'.\tag{$\dagger$}
\end{equation*}
The coefficient of $\tilde{T}_xC_u'$ on the left hand side is 
$p_{xu,syv}^*$. Assume first that $sx<x$. Then, arguing as in the proof 
of Proposition~\ref{recursp}, we find that the coefficient of $\tilde{T}_x
C_u'$ on the right hand side of ($\dagger$) is 
\[ p_{sxu,yv}^*+\varepsilon^{L(s)}p_{xu,yv}^*-\tilde{p}_{xu,syv}^s.\] 
Now we note that 
\[ \tilde{p}_{xu,syv}^s=\sum_{\atop{z \in X, w \in W'}{szw<zw<yv}} 
p_{xu,zw}^*\, M_{zw,yv}^s=M_{xu,yv}^s+\pi_{xu,yv}^s.\]
Thus, we conclude that 
\[M_{xu,yv}^s-\varepsilon^{L(s)}p_{xu,yv}^*+\pi_{xu,yv}^s=
\tilde{p}_{xu,syv}^s-\varepsilon^{L(s)}p_{xu,yv}^*=p_{sxu,yv}^*-
p_{xu,syv}^*.\]
This expression lies in $A_{<0}$; thus, we have shown that 
($\text{M1}^\prime$) holds. On the other hand, if $sx \not\in X$ and $tu<u$ 
(where $t=x^{-1}sx\in S'$), then the coefficient of $\tilde{T}_xC_u'$ on 
the right hand side of ($\dagger$) is 
\[(\varepsilon^{L(s)}+\varepsilon^{-L(s)}) p_{xu,yv}^*-\tilde{p}_{xu,syv}^s 
+ p_{xtu,yv}^*+\sum_{\atop{w \in W'}{u<w<tw}} M_{u,w}^tp_{xw,yv}^*.\]
Then a similar argument shows that ($\text{M1}^{\prime\prime}$) holds. 
\end{proof}

\begin{exmp} \label{muexp} Let $x,y \in X$, $u,v \in W'$ and $s \in S$ be 
such that $L(s)>0$ and $sxu<xu<yv<syv$. Assume now that $x=y$. First of all,
this forces that $u<v$ and $\pi_{xu,yv}^s=0$. Furthermore, we must
have $sx=sy\not\in X$. Thus, condition ($\text{M1}^{\prime\prime}$) in
Lemma~\ref{lemmue} yields that 
\[ M_{xu,yv}^s -\varepsilon^{L(s)}p_{xu,xv}^*- \sum_{\atop{w \in 
W'}{u<w<tw}} M_{u,w}^t p_{xw,xv}^*\in A_{<0}.\]
Now, recall that for any $w \in W$, we have $p_{xw,xv}^*=0$ unless $w=v$. 
It follows that 
\[ M_{xu,xv}^s=M_{u,v}^t \qquad \mbox{(where $t=x^{-1}sx \in S'$)}.\]
(This shows, in particular, that we can have $M_{xu,yv}^s\neq 0$ even
if $p_{xu,yv}^*=0$.)
\end{exmp}

\begin{rem} \label{muequal} Assume that $\Gamma=\Z$ and $L(s)=1$ for 
all $s \in S$ (equal parameter case). Then $A$ is the ring of Laurent 
polynomials in the indeterminate $\varepsilon$. Let $s \in S$, $x,y 
\in X$ and $u,v \in W'$ be such that $sxu<xu<yv<syv$. Then 
\begin{align*}
M_{xu,yv}^s&=\mbox{ coefficient of $\varepsilon^{-1}$ in }
\left\{\begin{array}{cl} P_{u,v}^*& \mbox{ if $x=y$},\\
p_{xu,yv}^* & \mbox{ if $x\neq y$}.  \end{array}\right.\\
&=\mbox{ coefficient of $\varepsilon^{L(yv)-L(xu)-1}$ in }
\left\{\begin{array}{cl} P_{u,v}& \mbox{ if $x=y$},\\
p_{xu,yv} & \mbox{ if $x\neq y$}.  \end{array}\right.
\end{align*}
This is easily seen using the formulae in Lemma~\ref{lemmue} and
Example~\ref{muexp}; note also that Remark~\ref{eqp} already shows 
that $M_{xu,yw}^s\in \Z$ in this case.
\end{rem}

\begin{cor} \label{recursp1} Let $\fC'$ be a left cell of $W'$. 
Then we have the following recursion formulae for $p_{xu,yv}^*$ where
$x,y\in X$ and $u,v \in \fC'$. 
\begin{itemize}
\item[(a)] If $y=1$, then  
\[ \renewcommand{\arraystretch}{1.2} 
p_{xu,v}^*=\left\{\begin{array}{cl} 1 & \mbox{ if $x=1$ and $u=v$},
\\ 0 & \mbox{ otherwise}.\end{array}\right.\]
\item[(b)] Now assume that $y \neq 1$ and let $s \in S$ be such that 
$sy<y$. If $L(s)=0$, then
\[ \renewcommand{\arraystretch}{1.2} p_{xu,yv}^*=\left\{\begin{array}{cl} 
p_{sxu,syv}^* & \quad \mbox{if $sx \in X$},\\ p_{xtu,syv}^* & \quad 
\mbox{if $sx\not\in X$},  \end{array}\right.\]
where $t=x^{-1}sx\in S'$ (if $sx\not\in X$). If $L(s)>0$, then
\[\renewcommand{\arraystretch}{1.5} p_{xu,yv}^*=\left\{\begin{array}{cl} 
p_{sxu,syv}^*+\varepsilon^{L(s)}p_{xu,syv}^*-\tilde{p}_{xu,yv}^s & 
\mbox{ if $sx<x$},\\ \varepsilon^{-L(s)}p_{sxu,yv}^* & \mbox{ if $sx>x$, 
$sx \in X$},\\ 0 & \mbox{ if $sx \not\in X$, $tu>u$}, \\
\begin{array}{l} \displaystyle (\varepsilon^{L(s)}+\varepsilon^{-L(s)})
p_{xu,syv}^*-\tilde{p}_{xu,yv}^s \\ \displaystyle  \quad 
+\underbrace{p_{xtu,syv}^*}_{\text{only if $tu\in \fC'$}}+
\sum_{\atop{w \in \fC'}{u<w<tw}} M_{u,w}^tp_{xw,syv}^* \end{array}& 
\mbox{ if $sx \not\in X$, $tu<u$}, \end{array}\right.\]
where $t=x^{-1}sx\in S'$ (if $sx\not\in X$) and 
\[\tilde{p}_{xu,yv}^s:=\sum_{\atop{z \in X, w \in \fC'}{x \leq z \leq
sy \text{ and } zw<zw<syv}} p_{xu,zw}^*\, M_{zw,syv}^s.\]
\end{itemize}
\end{cor}

\begin{proof} This immediately follows from Proposition~\ref{recursp}
and the following facts. Let $\leq_{\cL}'$ be the Kazhdan--Lusztig left 
pre-order relation on $W'$. By \cite[Prop.~3.3]{myind}, we have 
$p_{xu,yv}^*=0$ unless $xu=yv$ or $x<y$ and $u \leq_{\cL}' v$; furthermore, 
by \cite[\S 4]{myind}, we have the implication ``$xu \leq_{\cL} yv 
\Rightarrow u \leq_{\cL}' v$''. 
\end{proof}

\begin{algo} \label{varalgo} The following inductive procedure produces 
the partition of $W$ into left cells and the associated $W$-graphs. 

If $W=\{1\}$, then $\fC=\{1\}$ is the only left cell and there is
a canonical $W$-graph associated with it. Now assume that $W \neq \{1\}$
and let $W'\subsetneqq W$ be a proper parabolic subgroup. By induction,
we obtain the partition $W'=\fC_1' \amalg \ldots \amalg \fC_k'$ of $W'$ 
into left cells and the associated $W$-graphs. Now fix $i \in \{1,\ldots,k\}$.
Then, by the main result of \cite{myind}, the set $X\fC_i'$ is a union
of left cells, that is, we have
\[ X\fC_i'=\fC_{i,1} \amalg \fC_{i,2} \amalg \ldots \amalg \fC_{i,k_i}\]
where $\fC_{i,1},\fC_{i,2}\ldots,\fC_{i,k_i}$ are left cells of $W$ and
$X$ is the set of distinguished left coset representatives of $W'$ in $W$.
These left cells are determined as follows. By Corollary~\ref{recursp1} 
and Lemma~\ref{lemmue}, there is a recursive algorithm for simultaneously 
computing the polynomials 
\[ \{ p_{xu,yv}^* \} \quad \mbox{and} \quad \{ M_{xu,yv}^s\}
\qquad (x,y \in X, \; u,v \in \fC_i').\]
(The computations will only involve the elements in the fixed set 
$X\fC_i'$.) Once this is achieved, the set $X\fC_i'$ is decomposed
into left cells by the procedure in Example~\ref{explcell} (that is,
by explicitly working out the transitive closure of the relation 
$\leftarrow_{\cL}$); this also yields automatically the associated 
$W$-graphs. By letting $i$ run over all indices in $\{1,\ldots,k\}$ we 
eventually obtain all the left cells of $W$ and the associated $W$-graphs.
\end{algo}

\begin{table}[htbp] \caption{Examples of computations of left cells
(equal parameters)} \label{explc} {\small
\begin{center}
$\begin{array}{c@{\hspace{1mm}}c@{\hspace{1mm}}c@{\hspace{1mm}}
c@{\hspace{1mm}}c@{\hspace{2mm}}c@{\hspace{2mm}}c@{\hspace{2mm}}c} \hline 
W & |W| & W' & |X| & \max\{|X\fC'|\} & \# \text{left cells} 
& \# (\text{left cells}/{\approx}) & \text{time}\\ \hline
I_2(5) &   10 & A_1 & 5 & 5 & 4 & 4 & \text{0.01s} \\
H_3 &     120 & I_2(5) & 12 & 48 & 22 & 15 & \text{0.2s}  \\
D_4 &     192 & A_3 &8 & 24 & 36 & 12 & \text{0.1s}\\  
F_4 &    1152 & B_3 &24  &120 & 72 & 29  & \text{1s} \\
D_5 &    1920 & D_4 &10& 140 & 126 & 16 & \text{0.7s} \\
H_4 &   14400 & H_3 & 120 & 960 & 206 & 90  & \text{370s}\\
D_6 &   23040 & D_5 & 12&300 &  578 & 34 &  \text{10s}\\
E_6 &   51840 & D_5 & 27 & 675& 652 & 21 & \text{45s} \\
D_7 &   322560 & D_6 & 14 &1190 & 2416 & 49 & \text{190s}\\
A_8 &   362880 & A_7 & 9 & 810 & 2620 & 30 & \text{140s}\\ 
E_7 &  2903040 & D_6 & 126 &10710 & 6364 & 56 &\text{4h}\\ 
D_8 &  5160960 & D_7 & 16 & 3696& 11504 & 90 & \text{4h}\\
B_8 & 10321920 & B_7 & 16 & 8848 & 15304 & 346 & \text{58h}\\ \hline
\multicolumn{8}{l}{\text{($B_8$ requires 9GB main memory; in all 
other cases, 4GB are sufficient.)}}
\end{array}$
\end{center}}
\end{table}

In {\sf PyCox}, the function {\tt klcells} implements the procedure in 
Algorithm~\ref{varalgo}. As may be expected this leads to significant 
efficiency improvements compared with the use of ordinary Kazhdan--Lusztig 
polynomials (where the recursion involves {\it all}
elements of $W$). In the equal parameter case, one can apply some further 
simplications to reduce the number of left cells that have to be 
``induced'' from $W'$ to $W$: First of all, it is sufficient to induce 
only one left cell from each pair of left cells which are related by 
multiplication with the longest element in $W'$. (This follows from 
Yin \cite{yin1}.) Furthermore, assume that $i_1,i_2 \in \{1,\ldots,k\}$ 
are such that $\fC_{i_1}' \approx \fC_{i_2}'$ in the sense of 
Definition~\ref{eqcell}. Then, by \cite[Cor.~3.10]{myrel}, it is known 
that, for a suitable labelling, we have $k_{i_1}=k_{i_2}$ and 
$\fC_{i_1,j} \approx \fC_{i_2,j}$ for all $j \in \{1,\ldots,k_{i_1}\}$. 
Thus, it is sufficient to induce one left cell from each orbit under 
the star operations in Example~\ref{star}, and then to apply the star
operations to the resulting cells of $W$. This leads again to an 
enormous gain in efficiency. For example, in the computation of the 
left cells for type $E_7$, we only need to induce $34$ (instead of a 
total of $578$) left cells from a parabolic subgroup of type $D_6$; 
see Table~\ref{explc}. The efficiency also depends on the choice of
$W'$. For example, in type $E_7$ it is more efficient to use $W'$ of
type $D_6$ than of type $E_6$; in all other cases, we have chosen $W'$ 
such that the index $|W:W'|$ is as small as possible. Finally note that, if 
one is only interested in the partition of the group into left cells, then 
there are further techniques available; see, for example, Chen--Shi 
\cite{chenshi}.

One of the main advantages of being able to compute left cells and the 
corresponding $W$-graphs in a language like {\sf Python} lies in the fact 
that it provides immediate functionality for further handling of the data. 
We shall see a concrete example of this in the next section. 

\section{Leading coefficients of character values} \label{secJ}

We keep the general setting  of the previous sections; we assume now
that $W$ is finite and let $R=\R$. It is known that this is a splitting 
field for $W$ (see \cite[6.3.8]{gepf}). Let $\Irr(W)$ denote the set of 
simple $\R[W]$-modules (up to isomorphism). Let $K$ be the field of 
fractions of $A$ and $\cH_K=K \otimes_A \cH$. Then it is known that 
$\cH_K$ is split semisimple and abstractly isomorphic to $K[W]$ (see 
\cite[9.3.5]{gepf}); furthermore, the map $\varepsilon^g \mapsto 1$ 
($g\in\Gamma$) induces a bijection between $\Irr(\cH_K)$ and $\Irr(W)$ 
(see \cite[8.1.7]{gepf}). Given $E \in \Irr(W)$, we denote by 
$E_\varepsilon$ the corresponding irreducible representation of $\cH_K$.
It is known that 
\[\mbox{trace}(\tilde{T}_w,E_\varepsilon) \in \R[\Gamma] \qquad
\mbox{for all $w \in W$}\]
(see \cite[9.3.5]{gepf}). Thus, we can define 
\[ \ba_E:=\min\{g \in \Gamma_{\geq 0} \mid \varepsilon^g\, 
\mbox{trace}(\tilde{T}_w,E_\varepsilon) \in \R[\Gamma_{\geq 0}]
\mbox{ for all $w \in W$}\}.\]
Consequently, there are unique numbers $c_{w,E} \in \R$ 
($w \in W$) such that 
\[\varepsilon^{\ba_E}\,\mbox{trace}(\tilde{T}_w,E_\varepsilon)
=(-1)^{l(w)}\,c_{w,E}+\mbox{``higher terms''},\]
where ``higher terms'' means an $\R$-linear combination of terms 
$\varepsilon^g$ where $g>0$. These numbers are the ``leading coefficients 
of character values'', as defined and studied by Lusztig \cite{LuBook}, 
\cite{Lu4}, \cite{Lusztig03}. Since $\mbox{trace}(\tilde{T}_w,
E_\varepsilon)=\mbox{trace}(\tilde{T}_{w^{-1}},E_\varepsilon)$ for all 
$w \in W$ (see \cite[8.2.6]{gepf}), we certainly have 
\[ c_{w,E}=c_{w^{-1},E} \qquad \mbox{for all $w \in W$}.\]
Given $E$, there is at least one $w \in W$ such that $c_{w,E}\neq 0$
(by the definition of $\ba_E$). Hence, the sum of all $c_{w,E}^2$ 
($w\in W$) will be strictly positive and so we can write that sum as
$f_E\,\dim E$ where $f_E\in \R$ is strictly positive. In fact, we have 
the following orthogonality relations (see \cite[Exc.~9.8]{gepf}): 
\[ \sum_{w \in W} c_{w,E}\, c_{w,E'}= \left\{\begin{array}{cl} 
f_E \dim E & \quad \mbox{if $E \cong E'$},\\ 0 & \quad \mbox{otherwise}.  
\end{array}\right.\]
The connection with left cells is given by the following result,
first proved by Lusztig \cite[5.7]{LuBook}, \cite[3.5]{Lu4} in the equal 
parameter case (where the proof ultimately relies upon a geometric 
interpretation of the basis $\{C_w'\}$ of $\cH$); the general case (where
no geometric interpretation is available) is proved by an elementary 
argument in \cite[3.5, 3.8]{mylast}. Given $E \in \Irr(W)$ and a left cell 
$\fC$ of $W$, we denote by $m(\fC,E)$ the multiplicity of $E$ as an 
irreducible constituent of the left cell module $[\fC]_1$ (as defined in 
Remark~\ref{rwgr}).

\begin{prop} \label{prop41} Let $E \in \Irr(W)$ and $\fC$ be a
left cell. 
\begin{itemize}
\item[(a)] Let also $E' \in \Irr(W)$. Then 
\[ \sum_{w \in \fC} c_{w,E}\,c_{w,E'}=\left\{\begin{array}{cl} f_E\,
m(\fC,E) & \quad \mbox{if $E \cong E'$},\\ 0 & \quad 
\mbox{otherwise}.\end{array} \right.\]
\item[(b)] If $c_{w,E}\neq 0$ for some $w \in \fC$, 
then we also have $w^{-1}\in \fC$.
\end{itemize}
\end{prop}

In what follows, it will be important to renormalise the 
leading coefficients. In the equal parameter case, this renormalisation
is suggested by the formula in \cite[3.5(b)]{Lu4} (see Remark~\ref{rem44} 
below). In the unequal parameter case, we cannot just take the analogous 
formula; instead, we proceed as follows where we partly rely on a 
conjectural property. Following \cite[\S 1.5]{geja}, we define real 
numbers
\[\breve{n}_w:=\sum_{E \in \Irr(W)} f_E^{-1}\, c_{w,E} \qquad
\mbox{for any $w \in W$}.\]
(Note that, in \cite[\S 1.3]{geja}, we have omitted the factor
$(-1)^{l(w)}$ in the definition of $c_{w,E}$; hence, the numbers 
$\tilde{n}_w$ in \cite[\S 1.5]{geja} will be equal to $(-1)^{l(w)}
\breve{n}_w$.) With this notation, we can now state:

\begin{conj} \label{conj42} Let $\tilde{\cD}:=\{w \in W \mid \breve{n}_w 
\neq 0\}$. Then the following hold.
\begin{itemize}
\item[(a)] Every left cell of $W$ contains a unique element of $\tilde{\cD}$.
\item[(b)] We have $w^2=1$ and $\breve{n}_w=\pm 1$ for every $w \in 
\tilde{\cD}$.
\end{itemize}
\end{conj}

It is known that every left cell contains at least one element of 
$\tilde{\cD}$. (This follows from \cite[1.8.5 and 2.1.20]{geja}.) We 
expect that $\tilde{\cD}$ is precisely the set $\cD$ defined in 
Definition~\ref{disti} and that $\breve{n}_w= n_w$ for all $w \in \cD$. 
The advantage of the definition of $\tilde{\cD}$ is that this set can 
actually be computed in an efficient way; see Algorithm~\ref{algoB} below.

\begin{rem} \label{rem43} Conjecture~\ref{conj42} and the equality
$\tilde{\cD}=\cD$ are known to hold if Lusztig's properties 
{\bf P1}--{\bf P15} in \cite[14.2]{Lusztig03} are satisfied for $W,L$
(see \cite[\S 2.3]{geja} for details). By \cite[\S 16]{Lusztig03} (see 
\cite{Fokko} for $W$ of non-crystallographic type), {\bf P1}--{\bf P15} do 
hold in the equal parameter case where $\Gamma=\Z$ and $L(s)=1$ for all 
$s \in S$. It is also known that then the coefficients of the polynomials 
$P_{y,w}^*$ are non-negative; see \cite{Al}, \cite{Lu1}. Hence, in this 
case, we have 
\[\breve{n}_w=n_w=1 \qquad \mbox{for all $w \in \tilde{\cD}$}.\] 
We shall consider some cases of unequal parameters in the examples
below. 
\end{rem}

\begin{defn} \label{defct} Assume that Conjecture~\ref{conj42} holds for
$W,L$. Let $w \in W$ and $d \in \tilde{\cD}$ be the unique element
such that $w,d$ belong to the same left cell. Then we set 
\[ c_{w,E}^*:=(-1)^{l(w)+l(d)}\breve{n}_d\,c_{w,E} \qquad 
\mbox{for all $E \in \Irr(W)$}.\]
\end{defn}

\begin{rem} \label{rem44} Assume that we are in the equal parameter case
where $\Gamma=\Z$ and $L(s)=1$ for all $s \in S$. Let us check that then 
our renormalisation corresponds to the formula in \cite[3.5(b)]{Lu4}. Thus,
we claim that 
\begin{equation*}
c_{w,E}^*=(-1)^{l(w)+\ba_E}\,c_{w,E} \qquad \mbox{for all $w \in W$
and $E \in \Irr(W)$}.\tag{a}
\end{equation*}
This is seen as follows. Let $w \in W$. By Remark~\ref{rem43}, we have
$\breve{n}_d=1$ where $d \in \tilde{\cD}$ is the unique element such 
that $w,d$ belong to the same left cell. Hence, it will be enough to 
show that
\begin{equation*}
l(d) \equiv \ba_E \bmod 2 \quad \mbox{for all $E \in \Irr(W)$ such that
$c_{w,E} \neq 0$}.\tag{b}
\end{equation*}
Now let $E \in \Irr(W)$ be such that $c_{w,E} \neq 0$. Then, by 
\cite[3.3]{Lu4}, we have $\ba_E=\ba(w)$ where $z\mapsto \ba(z)$ 
($z \in W$) is the function defined by Lusztig \cite{Lu1}. This function 
is constant on the left cells of $W$ and so $\ba_E=\ba(d)$.  Thus, it 
remains to show that $l(d) \equiv \ba(d) \bmod 2$.  But this immediately
follows from \cite[3.2]{Lu1} (see also \cite[Rem.~2.3.5]{geja})
and property {\bf P5} in \cite[14.2]{Lusztig03}.~--~An explanation for
the renormalisation in (a) can be given by using the asymptotic algebra
$J$ introduced by Lusztig \cite{Lu2}. This algebra has a basis $\{t_w 
\mid w \in W\}$ and one can easily check that the map $t_w \mapsto 
(-1)^{l(w)+l(d)}t_w$ (where $d \in \tilde{\cD}$ is such that $w,d$ belong
to the same left cell) defines an algebra automorphism of~$J$. 
\end{rem}

\begin{rem} \label{rem44a} Assume that Conjecture~\ref{conj42} holds
for $W,L$. Let $\fC$ be a left cell and  consider the unique element
$d \in \tilde{\cD} \cap \fC$. Then we have:
\[ c_{d,E}^*=m(\fC,E) \qquad \mbox{for all $E \in \Irr(W)$}.\]
In the framework of Lusztig's theory of the asymptotic algebra, the 
above statement appears in \cite[21.4]{Lusztig03}; see also 
\cite[Chap.~12]{LuBook}. One can give a more elementary argument, as
follows. We consider the algebra $\tilde{J}$ defined in 
\cite[\S 1.5]{geja}. Using $\tilde{J}$, one can define a partition of $W$
into ``left $\tilde{J}$-cells''; see \cite[\S 1.6]{geja}. By 
\cite[Prop.~2.1.20]{geja}, every Kazhdan--Lusztig left cell is a union of 
left $\tilde{J}$-cells. Hence, by \cite[Lemma~3.7]{mylast} and 
\cite[Exp.~1.8.5]{geja}, we have
\[ m(\fC,E)=\sum_{d \in \tilde{\cD} \cap \fC} \breve{n}_d c_{d,E}
\qquad \mbox{for all $E \in \Irr(W)$}.\]
Thus, the claim immediately follows from the assumption that 
Conjecture~\ref{conj42} holds. In particular, we have the following
formula for the decomposition of the left cell module $[\fC]_1$:
\[ [\fC]_1=\sum_{E \in \Irr(W)} m(\fC,E)\,E=\breve{n}_d\sum_{E \in 
\Irr(W)} c_{d,E}\, E\]
(in the appropriate Grothendieck group of representations).
\end{rem}

\begin{defn} \label{rem45} Assume that Conjecture~\ref{conj42} holds
for $W,L$. Let $\fC$ be a left cell of $W$ and denote by $\Irr(W\mid \fC)$
the set of all $E \in \Irr(W)$ such that $E$ is an irreducible constituent 
of $[\fC]_1$. Then we define 
\[ \fX(W\mid \fC):=\bigl(c_{w,E}^*\bigr)_{E \in 
\Irr(W\mid \fC), \, w \in \fC \cap \fC^{-1}}.\]
(Following Lusztig \cite{Lu4}, \cite{Lusztig03}, this table can be 
interpreted as the character table of the subalgebra of the asymptotic 
algebra $J$ which is spanned by $t_w$ for $w \in \fC \cap \fC^{-1}$; the 
unique element $d \in \tilde{\cD} \cap \fC$ corresponds to the identity 
element of this algebra, in accordance with Remark~\ref{rem44a}.)
Note that, by Proposition~\ref{prop41}, we have $E \in \Irr(W\mid 
\fC)$ if and only if $c_{w,E} \neq 0$ for some $w \in \fC$; furthermore, 
$c_{w,E}=0$ unless $w,w^{-1}$ belong to the same left cell. Thus, every
non-zero leading coefficient will appear in one of the tables 
$\fX(W\mid \fC)$ as $\fC$ runs over the left cells of $W$.
\end{defn}

\begin{exmp} \label{ctweyl} Assume that $W$ is a finite Weyl group and that
we are in the equal parameter case where $\Gamma=\Z$ and $L(s)=1$ for all
$s \in S$. Then the tables $\fX(W\mid \fC)$ have been determined explicitly 
by Lusztig \cite[3.14]{Lu4}, based on the results in \cite{LuBook}. In 
particular, it turns out that, if $E\in \Irr(W)$ is ``special'' in the 
sense of Lusztig \cite{Lusztig79b}, then $c_{w,E}^*\geq 0$ for all $w \in W$. 
(Except for some exceptional cases in type $E_7$ and $E_8$, the latter 
statement already appeared in \cite[Prop.~7.1]{LuBook}; one can also check 
this property directly in the exceptional cases by using the methods in the 
proof of \cite[Prop.~7.1]{LuBook}.) Furthermore, still assuming that $E$ is 
special, we actually have $c_{w,E}^*>0$ for all $w \in \fC \cap \fC^{-1}$ 
where $\fC$ is a left cell such that $m(\fC,E)>0$. Thus, for any given 
left cell $\fC$, all the entries in the row of $\fX(W\mid\fC)$ 
corresponding to the unique special representation occurring in $[\fC]_1$ 
are strictly positive. Note that, by Proposition~\ref{prop41}, there
can be at most one row with this property.
\end{exmp}

We shall now be interested in computing the tables $\fX(W\mid \fC)$ 
explictly in the case where $W$ is not of crystallographic type and also 
in some examples involving unequal parameters. 

\begin{algo} \label{algoB} The following procedure verifies if
Conjecture~\ref{conj42} holds for $W,L$ and determines the tables 
$\fX(W\mid \fC)$ for all left cells of $W$.

{\sc Step 1}. Let $\mbox{Cl}(W)$ be the set of conjugacy classes of $W$.
Using the inductive description in \cite[Prop.~8.2.7]{gepf}, we determine 
the ``class polynomials'' $f_{w,C} \in A$ for all $w \in W$ and all 
$C \in \mbox{Cl}(W)$. These polynomials have the following property. 
For $w \in W$, define $T_w:=\varepsilon^{L(w)}\tilde{T}_w$; for any 
$C \in \mbox{Cl}(W)$ let $d_{\text{min}}(C)=\min\{l(w) \mid w 
\in C\}$ and let $w_C \in C$ be a representative such that $l(w_C)=
d_{\text{min}}(C)$. Then we have:
\[ \mbox{trace}(T_w,E_\varepsilon)=\sum_{C\in \text{Cl}(W)} f_{w,C} \, 
\mbox{trace}(T_{w_C},E_\varepsilon)\qquad\mbox{for all $E\in \Irr(W)$}.\]

{\sc Step 2}. By \cite[Chap.~10, 11]{gepf}, the character tables 
\[ X(\cH)=\bigl(\mbox{trace}(T_{w_C}, E_\varepsilon)\bigr)_{E \in \Irr(W),
\, C\in \text{Cl}(W)}\]
are explicitly known. Furthermore, the functions $E \mapsto \ba_E$ and 
$E \mapsto f_E$ are explicitly known; see, for example, the appendix of 
\cite{gepf} (equal parameter case) and the summary in \cite[\S 1.3]{geja} 
for unequal parameters. Thus, in combination with the 
class polynomials in Step~1, we can explicitly compute all the leading 
coefficients $c_{w,E}$ where $w \in W$ and $E \in \Irr(W)$. Consequently, 
we can then also compute the numbers $\breve{n}_w$ for all $w \in W$,
and the set $\tilde{\cD}$.

{\sc Step 3}. By Algorithm~\ref{varalgo}, we can determine the 
partition of $W$ into left cells. (We do not need the additional 
information on the associated $W$-graphs here.) Let $\fC$ be a fixed left 
cell. Using the data in Step 2, we can then explicitly verify if 
Conjecture~\ref{conj42} holds. Using the formula in 
Proposition~\ref{prop41}(a), we can find the multiplicities 
$m(\fC,E)$ for all $E \in \Irr(W)$. Thus, the table $\fX(W\mid \fC)$
is determined.
\end{algo}

In {\sf PyCox}, the function {\tt leftcellleadingcoeffs} implements the 
procedure in Algorithm~\ref{algoB} for a given left cell. This allows the 
explicit determination of all the tables $\fX(W \mid \fC)$ for groups
$W$ of rank up to around $7$ and any weight function $L$. All this even 
works for type $E_7$ where it takes about 3 hours and requires 4GB of 
main memory. With this information, it is then straightforward to verify 
Kottwitz's conjecture for type $E_7$, as mentioned in the introduction. 

Performing only Steps~1 and~2 of Algorithm~\ref{algoB} yields the set 
$\tilde{\cD}$ and all the leading coefficients $c_{w,E}$. This even works 
for type $E_8$ where it takes nearly 18 days and requires about 22GB of main 
memory to compute the $101796$ elements in $\tilde{\cD}$. (As far as I am 
aware, these elements have not been explicitly known before.) All the 
known sets $\tilde{\cD}$ for $W$ of exceptional type are explicitly 
stored in a compact format within {\sf PyCox};  see the function 
{\tt libdistinv}.

The explicit data in the examples below have been computed
with the help of the {\sf PyCox} function {\tt leftcellleadingcoeffs}.

\begin{sidewaystable}[ph!]
\caption{The tables $\fX(W\mid \fC)$ for big left cells in type $H_4$;
$\alpha=(1+\sqrt{5})/2$} 
\label{cell1}
\begin{center} 
{$\arraycolsep 1.3pt 
\begin{array}{c}\begin{array}{ccccccccccccccc} \hline
\multicolumn{15}{l}{\text{Left cells with $326$ elements}} \\\hline
8_r&1&0&{-}1&1&0&0&{-}1&{-}1&0&0&1&{-}1&0&1\\
8_{rr}&1&{-}1&0&1&{-}1&{-}1&0&0&1&1&{-}1&0&1&{-}1\\
18_r&1&1&0&{-}1&{-}1&1&0&0&1&{-}1&{-}1&0&1&1\\
24_t&1&2{-}\alpha&4{-}3\alpha&5{-}3\alpha&2{-}\alpha&2{-}\alpha&0&0&{-}2{+}
\alpha& {-}2{+}\alpha&{-}5{+}3\alpha&{-}4{+}3\alpha&{-}2{+}\alpha&{-}1\\
\overline{24}_t&1&1{+}\alpha&1{+}3\alpha&2{+}3\alpha&1{+}\alpha&1{+}\alpha&0&
0& {-}1{-}\alpha&{-}1{-}\alpha&{-}2{-}3\alpha&{-}1{-}3\alpha&{-}1{-}\alpha
&{-}1\\
24_s&1&2{-}2\alpha&7{-}4\alpha&13{-}8\alpha&6{-}4\alpha&16{-}10\alpha&7{-}4
\alpha& 7{-}4\alpha&16{-}10\alpha&6{-}4\alpha&13{-}8\alpha&7{-}4\alpha&2{-}2
\alpha&1\\
\overline{24}_s&1&2\alpha&3{+}4\alpha&5{+}8\alpha&2{+}4\alpha&6{+}10\alpha&
  3{+}4\alpha&3{+}4\alpha&6{+}10\alpha&2{+}4\alpha&5{+}8\alpha&3{+}4\alpha&
2\alpha&1\\
30_s&1&{-}1{+}\alpha&1{-}\alpha&{-}2{+}\alpha&3{-}\alpha&1{-}\alpha&{-}2{+}
2\alpha&{-}2{+}2\alpha& 1{-}\alpha&3{-}\alpha&{-}2{+}\alpha&1{-}\alpha&{-}
1{+}\alpha&1\\
\overline{30}_s&1&{-}\alpha&\alpha&{-}1{-}\alpha&2{+}\alpha&\alpha&{-}2
\alpha&{-}2\alpha&\alpha&2{+}\alpha&{-}1{-}\alpha&\alpha&{-}\alpha&1\\
40_r&1&2&3&1&0&{-}2&{-}1&{-}1&{-}2&0&1&3&2&1\\
48_{rr}&2&0&2&{-}2&0&0&0&0&0&0&2&{-}2&0&{-}2\\\hline
\end{array} \\\\
\begin{array}{ccccccccccccccccccc}\hline \multicolumn{19}{l}{\text{Left cells
with $392$ elements}} \\ \hline
 10_r&1&0&{-}1&{-}1&1&0&0&0&{-}1&0&1&0&0&1&{-}1&{-}1&0&1\\
 16_t&1&{-}1{+}\alpha&{-}\alpha&1{-}\alpha&1&{-}1&{-}1&\alpha&0&
  {-}\alpha&0&1&1&{-}1&\alpha&{-}1{+}\alpha&1{-}\alpha&{-}1\\
 \overline{16}_t&1&{-}\alpha&{-}1{+}\alpha&\alpha&1&{-}1&{-}1&
  1{-}\alpha&0&{-}1{+}\alpha&0 &1&1&{-}1&1{-}\alpha &{-}\alpha&\alpha&{-}1\\
 18_r&1&0&1&{-}1&{-}1&0&0&0&1&0&1&0&0&{-}1&1&{-}1&0&1\\
 24_t&1&3{-}2\alpha&3{-}2\alpha&1{-}\alpha&5{-}3\alpha&2{-}\alpha&2{-}\alpha&
  1{-}\alpha&0&{-}1{+}\alpha&0&{-}2{+}\alpha&{-}2{+}\alpha&{-}5{+}3\alpha&
  {-}3{+}2\alpha& {-}1{+}\alpha&{-}3{+}2\alpha&{-}1\\
 \overline{24}_t&1&1{+}2\alpha&1{+}2\alpha&\alpha&2{+}3\alpha&1{+}\alpha&
  1{+}\alpha&\alpha&0&{-}\alpha& 0&{-}1{-}\alpha&{-}1{-}\alpha&
  {-}2{-}3\alpha&{-}1{-}2\alpha& {-}\alpha&{-}1{-}2\alpha&{-}1\\
 24_s&1&4{-}3\alpha&5{-}3\alpha&2{-}\alpha&13{-}8\alpha&7{-}4\alpha&
  7{-}4\alpha& 11{-}7\alpha&16{-}10\alpha& 11{-}7\alpha&2{-}2\alpha&
  7{-}4\alpha&7{-}4\alpha& 13{-}8\alpha&5{-}3\alpha&2{-}\alpha&4{-}3\alpha&1\\
 \overline{24}_s&1&1{+}3\alpha&2{+}3\alpha&1{+}\alpha&5{+}8\alpha&3{+}4\alpha&
  3{+}4\alpha& 4{+}7\alpha&6{+}10\alpha&4{+}7\alpha&2\alpha&3{+}4\alpha&
  3{+}4\alpha&5{+}8\alpha& 2{+}3\alpha&1{+}\alpha&1{+}3\alpha&1\\
 30_s&1&1&{-}1&2{-}\alpha&{-}2{+}\alpha&1{-}\alpha&1{-}\alpha&2{-}\alpha&
  {-}2{+}2\alpha& 2{-}\alpha& 2{-}2\alpha&1{-}\alpha&1{-}\alpha&
  {-}2{+}\alpha&{-}1&2{-}\alpha&1&1\\ 
  \overline{30}_s'&1&1&{-}1&1{+}\alpha&{-}1{-}\alpha&\alpha&\alpha&
  1{+}\alpha& {-}2\alpha&1{+}\alpha& 2\alpha&\alpha&\alpha&{-}1{-}\alpha&
  {-}1&1{+}\alpha&1&1\\
 40_r&2&1&3&3&2&{-}2&{-}2&{-}1&{-}2&{-}1&2&{-}2&{-}2&2&3&3&1&2\\
 48_{rr}&2&1&1&1&{-}2&0&0&1&0&{-}1&0&0&0&2&{-}1&{-}1&{-}1&{-}2
\\\hline \end{array} \\\\
\begin{array}{ccccccccccccccccccccccccc} \hline 
\multicolumn{25}{l}{\text{Left cells with $436$ elements}}\\ \hline
 6_s&1&\alpha&{-}1{-}\alpha&{-}\alpha&{-}\alpha&\alpha&1{+}\alpha&{-}1&{-}1&
  {-}\alpha& 1{+}\alpha& {-}\alpha&{-}1{-}\alpha&\alpha&\alpha&1&1{+}\alpha&
  \alpha&{-}1{-}\alpha& {-}\alpha&{-}\alpha&{-}1&1&\alpha\\
 \overline{6}_s&1&1{-}\alpha&{-}2{+}\alpha&{-}1{+}\alpha&{-}1{+}\alpha&
  1{-}\alpha&2{-}\alpha& {-}1&{-}1&{-}1{+}\alpha&2{-}\alpha&{-}1{+}\alpha&
  {-}2{+}\alpha&1{-}\alpha&1{-}\alpha&1&2{-}\alpha& 1{-}\alpha&
 {-}2{+}\alpha& {-}1{+}\alpha&{-}1{+}\alpha&{-}1&1&1{-}\alpha\\
 16_t&1&\alpha&{-}\alpha&{-}1&{-}1&{-}1{+}\alpha&1&1{-}\alpha&0&0&0&0&0&0&0&
  0&{-}1& 1{-}\alpha&\alpha&1&1&{-}1{+}\alpha&{-}1&{-}\alpha\\
 \overline{16}_t&1&1{-}\alpha&{-}1{+}\alpha&{-}1&{-}1&{-}\alpha&1&\alpha&0&
  0&0&0&0&0& 0&0&{-}1&\alpha&1{-}\alpha&1&1&{-}\alpha&{-}1&\alpha{-}1\\
 24_t&1&1{-}\alpha&3{-}2\alpha&2{-}\alpha&2{-}\alpha&3{-}2\alpha&5{-}3\alpha&
  1{-}\alpha&0&0&0&0&0&0&0&0&{-}5{+}3\alpha&{-}3{+}2\alpha&{-}3{+}2\alpha&
  {-}2{+}\alpha& {-}2{+}\alpha& {-}1{+}\alpha&{-}1&\alpha{-}1\\
 \overline{24}_t&1&\alpha&1{+}2\alpha&1{+}\alpha&1{+}\alpha&1{+}2\alpha&
  2{+}3\alpha& \alpha&0&0&0&0&0&0&0&0&{-}2{-}3\alpha&{-}1{-}2\alpha&
 {-}1{-}2\alpha&{-}1{-}\alpha& {-}1{-}\alpha&{-}\alpha&{-}1&{-}\alpha\\
 24_s&1&1{-}\alpha&5{-}3\alpha&3{-}2\alpha&3{-}2\alpha&8{-}5\alpha&
  13{-}8\alpha& 2{-}\alpha&6{-}4\alpha&10{-}6\alpha&6{-}4\alpha&
  10{-}6\alpha&16{-}10\alpha& 4{-}2\alpha&4{-}2\alpha& 2{-}2\alpha&
  13{-}8\alpha&8{-}5\alpha&5{-}3\alpha& 3{-}2\alpha&3{-}2\alpha&2{-}\alpha&
  1&1{-}\alpha\\
 \overline{24}_s&1&\alpha&2{+}3\alpha&1{+}2\alpha&1{+}2\alpha&3{+}5\alpha&
  5{+}8\alpha& 1{+}\alpha&2{+}4\alpha&4{+}6\alpha&2{+}4\alpha&4{+}6\alpha&
  6{+}10\alpha&2{+}2\alpha&2{+}2\alpha& 2\alpha&5{+}8\alpha&3{+}5\alpha&
  2{+}3\alpha&1{+}2\alpha&1{+}2\alpha&1{+}\alpha&1&\alpha\\
 30_s&2&1&1{-}\alpha&0&0&2{-}\alpha&{-}4{+}2\alpha&1{-}\alpha&
  {-}1{+}2\alpha&1{-}\alpha& 4{-}3\alpha&1{-}\alpha&\alpha&1{-}\alpha&
  1{-}\alpha&3{-}2\alpha&{-}4{+}2\alpha& 2{-}\alpha&1{-}\alpha&0& 0&
  1{-}\alpha&2&1\\
 \overline{30}_s&2&1&\alpha&0&0&1{+}\alpha&{-}2{-}2\alpha&\alpha&1{-}2\alpha&
  \alpha&1{+}3\alpha&\alpha&1{-}\alpha&\alpha&\alpha&1{+}2\alpha&
  {-}2{-}2\alpha& 1{+}\alpha&\alpha&0&0&\alpha& 2&1\\
 40_r&2&1&3&0&0&{-}1&2&3&0&{-}2&0&{-}2&{-}2&{-}2&{-}2&2&2&{-}1& 3&0&0&3&2&1\\
 48_{rr}&2&1&1&0&0&1&{-}2&1&0&0&0&0&0&0&0&0&2&{-}1& {-}1&0&0&{-}1&{-}2&{-}1\\
\hline \end{array}\end{array}$}
\end{center}
\end{sidewaystable}

\begin{exmp} \label{gammah34} Let $W$ be of type $H_3$ or $H_4$. 
Let $\fC$ be a left cell of $W$. Using Algorithm~\ref{varalgo}, we obtain
the left cells of $W$; we have 
\[ |\fC|\in \left\{\begin{array}{cl} \{1,5,6,8\} & \qquad \mbox{in type 
$H_3$},\\ \{1,8,18,25,32,36,326,392,436\} & \qquad \mbox{in type $H_4$};
\end{array}\right.\] 
(See also Alvis \cite{Al}.) If $|\fC|$ equals $1$, $5$, 
$25$ or $36$, then $[\fC]_1$ is irreducible and the table
$\fX(W\mid \fC)$ is $(1)$. Now assume that $|\fC|$ equals $6$, $8$, $18$ 
or $32$. Then $[\fC]_1=E_1\oplus E_2$ where $E_1 \not\cong E_2$, 
$\dim E_1=\dim E_2$ and where we choose the notation such that $E_1$ is a 
special representation. Then the table $\fX(W\mid \fC)$ is   
\[  \begin{array}{ccr} E_1 &1 & 1 \\ E_2 & 1 & -1\end{array} \qquad 
\mbox{or}\qquad \begin{array}{ccc} E_1 & 1 &  \alpha \\ E_2 & 1 & 
1-\alpha \end{array},\]
according to whether $f_{E_1}$ equals $2$ or $2+\alpha$, respectively, 
where $\alpha=\frac{1}{2}(1+\sqrt{5})$. Finally, if $|\fC|$ equals $326$, 
$392$ or $436$, then $\fX(W\mid \fC)$ is given by Table~\ref{cell1}. Here, 
we use the notation for $\Irr(W)$ defined in the appendix of \cite{gepf}. 
As in Example~\ref{ctweyl} we note that there is a row in which all entries
are strictly positive, and this row corresponds to the unique special 
representation occurring in $[\fC]_1$ (which is $\overline{24}_s$ in 
Table~\ref{cell1}).
\end{exmp}

\begin{exmp} \label{ctdi} Let $W$ be of type $I_2(m)$ where $m\geq 3$ and 
$S=\{s_1,s_2\}$. Assume that we are in the equal parameter case, where 
$\Gamma=\Z$ and $L(s)=1$ for all $s \in S$. Let $\zeta \in \C$ be a root 
of unity of order $m$, chosen such that $\zeta+\zeta^{-1}=2\cos(2\pi/m)$. 
By \cite[\S 5.4]{gepf}, we have
\[ \Irr(W)=\left\{\begin{array}{cl} \{1_W,\mbox{sgn},\sigma_1,
\sigma_2, \ldots, \sigma_{(m-1)/2}\}& \quad \mbox{if $m$ is odd},\\
\{1_W,\mbox{sgn},\sigma_1,\sigma_2, \ldots, \sigma_{(m-2)/2},
\mbox{sgn}_1,\mbox{sgn}_2\}& \quad \mbox{if $m$ is even},
\end{array}\right.\]
where $1_W$ is the unit and $\mbox{sgn}$ is the sign representation, all
$\sigma_j$ are $2$-dimensional, and $\mbox{sgn}_1,\mbox{sgn}_2$ are two
further $1$-dimensional representations when $m$ is even, in which case we
fix the notation such that $s_1$ acts as $+1$ in $\mbox{sgn}_1$ and as $-1$
in $\mbox{sgn}_2$. The left cells and the corresponding left cell modules 
are given as follows (see, for example, \cite[2.1.8, 2.2.8]{geja}):
\begin{gather*}
\{1_0\},\quad \{1_m\}, \quad \{2_1,1_2,2_3, \ldots,1_{m-1}\}, 
\quad \{1_1,2_2,1_3, \ldots,2_{m-1}\} \quad \mbox{($m$ odd)}\\
\{1_0\},\quad \{1_m\},\quad \{2_1,1_2,2_3,\ldots,
2_{m-1}\}, \quad \{1_1,2_2,1_3, \ldots,1_{m-1}\} \quad \mbox{($m$ even)}.
\end{gather*}
Here, for any $k \geq 0$, we write $1_k=s_1s_2s_1 \cdots $ ($k$ factors) 
and $2_k=s_2s_1s_2 \cdots$ ($k$ factors); note that $1_m=2_m$. We have:
\begin{gather*}
[1_0]_1 = 1_W,\quad [2_1,1_2,2_3,\ldots, 2_{m-1}]_1=(\mbox{sgn}_1) \oplus  
\mbox{(sum of all $\sigma_j$)},\\
[1_1,2_2,1_3, \ldots,1_{m-1}]_1=(\mbox{sgn}_2) \oplus  
\mbox{(sum of all $\sigma_j$)}, \quad [1_m]_1 =\mbox{sgn}.
\end{gather*}
where $\mbox{sgn}_1$ and $\mbox{sgn}_2$ have to be omitted if $m$ is odd.
(Note that \cite[2.2.8]{geja} contains a misprint: the roles of 
$\mbox{sgn}_1$, $\mbox{sgn}_2$ need to be changed there.)
By \cite[Exp.~1.3.7]{geja}, we have $\ba_{1_W}=0$ and 
$\ba_{\text{sgn}}=m$; all the other irreducible representations have 
$\ba$-invariant equal to~$1$. First of all, one easily checks that 
\[ c_{w,1_W}^*=\left\{\begin{array}{cl} 1 & \mbox{ if $w=1$},\\
0 & \mbox{ otherwise},\end{array}\right. \qquad \mbox{and} \qquad
c_{w,\text{sgn}}^*=\left\{\begin{array}{cl} 1 & \quad \mbox{if $w=w_0$},\\
0 & \quad \mbox{otherwise},\end{array}\right.\]
where $w_0\in W$ is the longest element. Next consider $\mbox{sgn}_1$
and $\mbox{sgn}_2$ (in case $m$ is even). Let $w \in W$. For $i=1,2$ we 
denote by $l_i(w)$ the number of occurrences of the generator $s_i$ in a 
reduced expression for $w$. Then 
\[ \mbox{trace}(\tilde{T}_w,\mbox{sgn}_1)=(-1)^{l_2(w)}
\varepsilon^{l_1(w)-l_2(w)}\]
and so
\[ c_{w,\text{sgn}_1}^*=\left\{\begin{array}{cl} -(-1)^{l_2(w)}
& \quad \mbox{if $l_1(w)-l_2(w)=-1$},\\ 0 & \quad \mbox{otherwise}.
\end{array}\right.\]
A similar formula holds for $c_{w,\text{sgn}_2}^*$ where the roles
of $l_1(w)$ and $l_2(w)$ need to be interchanged. Finally, consider 
$\sigma_j$. By \cite[Lemma~8.3.3]{gepf}, we have 
\[ \mbox{trace}(\tilde{T}_{s_i},\sigma_j)=\varepsilon-\varepsilon^{-1} 
\quad \mbox{and} \quad \mbox{trace}(\tilde{T}_{w_k},\sigma_j)=
\zeta^{jk}+\zeta^{-jk}\]
where $w_k=(s_1s_2)^k$ for $0 \leq k \leq m/2$. In particular, we see that 
\[ c_{s_1,\sigma_j}^*=c_{s_2,\sigma_j}^*=1 \quad \mbox{and} \quad 
c_{w_k,\sigma_j}^*=0 \quad \mbox{for all $0\leq k \leq m/2$}.\]
Let $y\in W$ be a conjugate of $s_1$ or $s_2$. Then $l(y)$ is odd and we 
write $l(y)=2k+1$ where $k \geq 0$. Assume that $k \geq 2$ and let 
$i \in \{1,2\}$ be such that $y'=s_iys_i<y$. Then $s_iy$ or $ys_i$ equals 
$w_k$. So we have 
\[\mbox{trace}(\tilde{T}_{y},\sigma_j)=\mbox{trace}(\tilde{T}_{y'},
\sigma_j) +(\varepsilon-\varepsilon^{-1}) \mbox{trace}(\tilde{T}_{w_k},
\sigma_j).\]
Since $\ba_{\sigma_j}=1$, this yields that $c_{y,\sigma_j}=
c_{y',\sigma_j} +(\zeta^{jk}+ \zeta^{-jk})$. Thus, we have 
\[ c_{y,\sigma_j}^*=c_{y,\sigma_j}=1+\sum_{1\leq i \leq k}
\bigl(\zeta^{ji}+ \zeta^{-ji}\bigr).\]
For example, for $m=5$, we obtain for the two left cells with $m-1=4$ 
elements:
\[\fX(W\mid \fC): \quad \begin{array}{ccc} \sigma_1 & 1 & \alpha
\\ \sigma_2 & 1 & 1-\alpha \end{array} \qquad \mbox{where} \qquad 
\alpha=\frac{1}{2}(1+\sqrt{5}).\]
Having computed all the leading coefficients for $W$, we also see that 
\[ \tilde{\cD}=\{1,s_1,s_2,w_0\}.\]
To conclude, let $E\in \Irr(W)$ be special, that is, $E \in \{1_W, 
\mbox{sgn},\sigma_1\}$. By the above computations, we see that 
$c_{w,E}^*\geq 0$ for all $w \in W$; note also that $c_{y,\sigma_1}^*>0$ 
where $l(y)=2k+1$ and $1\leq k \leq m/2-1$. Using this property and
the explicit description of the left cells, we deduce that $c_{w,E}^*>0$ 
for all $w\in \fC \cap \fC^{-1}$ where $\fC$ is a left cell with 
$m(\fC,E)>0$.
\end{exmp}

\begin{conj} \label{cspec} Assume that Conjecture~\ref{conj42} holds for
$W,L$ and define 
\[ \cS_L(W):=\{ E \in \Irr(W) \mid c_{w,E}^* \geq 0 \mbox{ for all 
$w \in W$}\}.\] 
Then, for each left cell $\fC$ of $W$, there is a unique $E \in \cS_L(W)$
such that $m(\fC,E)>0$; furthermore, for this $E$, we have $m(\fC,E)=1$ and 
$c_{w,E}^*>0$ for all $w \in \fC\cap \fC^{-1}$. 
\end{conj}

\begin{rem} \label{newchar} Let $W$ be a finite Coxeter group and assume 
that we are in the equal parameter case where $\Gamma=\Z$ and $L(s)=1$ for 
all $s\in S$. Then the above conjecture holds where $\cS_L(W)$ consists
precisely of the ``special'' representations as originally defined by
Lusztig \cite{Lusztig79b}.

Indeed, by standard reduction arguments, we can assume that $W$ is 
irreducible. If $W$ is a finite Weyl group, the assertion holds by the 
results of Lusztig \cite{LuBook}, \cite{Lu4}, as already discussed in 
Example~\ref{ctweyl}. If $W$ is of type $I_2(m)$, $H_3$ or $H_4$, then 
the required assertions are verified by inspection using the data in 
Examples~\ref{gammah34} and~\ref{ctdi}.
\end{rem}

\begin{rem} \label{countcell} Assume that Conjectures~\ref{conj42} and 
\ref{cspec} hold for $W,L$. Then we have
\[ \sum_{E \in \cS_L(W)} \dim E\,=\mbox{ number of left cells of $W$ (with
respect to $L$)}.\]
\end{rem}

\begin{proof} We consider the quantity
\[ \nu=\sum_{\fC} \sum_{E \in \cS_L(E)} m(\fC,E)\]
where the first sum runs over all left cells of $W$. Since the direct 
sum of all left cell modules $[\fC]_1$ is isomorphic to the regular 
representation of $W$, we have 
\[ \dim E= \sum_{\fC} m(\fC,E) \qquad \mbox{for every $E \in \Irr(W)$}.\]
This shows that $\nu=\sum_{E \in \cS_L(W)} \dim E$. On the other 
hand, by Conjecture~\ref{cspec}, we have 
\[ 1=\sum_{E\in \cS_L(W)} m(\fC,E) \qquad \mbox{for each left cell $\fC$}.\]
So $\nu$ equals the number of left cells. This yields the desired equality.
\end{proof}

Let us now consider some examples with unequal parameters.

\begin{exmp} \label{uni2m} Let $W$ be of type $I_2(m)$ where $m\geq 3$
is even and $S=\{s_1,s_2\}$. Assume that we have a weight function 
such that $b=L(s_1)>a=L(s_2)>0$. The left cells and the corresponding
left cell modules are given as follows (see, for example, \cite[2.1.8, 
2.2.8]{geja}):
\[ \{1_0\},\;\; \{2_1\},\;\; \{1_{m-1}\},\;\; \{1_m\},\;\;
\{1_1,2_2,1_3,\ldots,2_{m-2}\}, \;\;\{1_2,2_3,1_4,\ldots,2_{m-1}\}.\]
(Notation as in Example~\ref{ctdi}.) We have:
\begin{gather*}
[1_0]_1 = 1_W,\quad
[2_1]_1 =\mbox{sgn}_1,\quad
[1_{m-1}]_1  =\mbox{sgn}_2,\quad [1_m]_1 =\mbox{sgn},\\
[1_1,2_2,1_3,\ldots,2_{m-2}]_1=[1_2,2_3,1_4,\ldots,2_{m-1}]=
\mbox{sum of all $\sigma_j$}.
\end{gather*}
By \cite[Exp.~1.3.7]{geja}, the $\ba$-invariants are given as follows:
\begin{center}
$\ba_{1_W}=0, \quad \ba_{\text{sgn}_1}=a, \quad \ba_{\text{sgn}_2}=
\frac{m}{2}(b-a)+a, \quad \ba_{\text{sgn}}=\frac{m}{2}(a+b), 
\quad \ba_{\sigma_j}=b$.
\end{center}
Arguing as in Example~\ref{ctdi}, we find the following leading coefficients:
\begin{alignat*}{2}
c_{w,1_W}&=\left\{\begin{array}{cl} 1 & \mbox{ if $w=1_0$},\\
0 & \mbox{ otherwise},\end{array}\right. \quad &\mbox{and} \quad
c_{w,\text{sgn}}&=\left\{\begin{array}{cl} 1 & \mbox{ if $w=1_m$},\\
0 & \mbox{ otherwise}.\end{array}\right.\\
c_{w,\text{sgn}_1}&=\left\{\begin{array}{cl} 1 & \mbox{ if $w=2_1$},\\ 
0 & \mbox{ otherwise},\end{array}\right.\quad &\mbox{and}\quad
c_{w,\text{sgn}_2}&=\left\{\begin{array}{cl} -(-1)^{m/2} & 
\mbox{ if $w=1_{m-1}$},\\ 0 & \mbox{ otherwise}.\end{array}\right.
\end{alignat*}
For $\sigma_j$, we now obtain $c_{1_1,\sigma_j}=1$, $c_{2_1,\sigma_j}
=0$ and also $c_{w_k,\sigma_j}=0$ where $w_k=(s_1s_2)^k$ for $0 
\leq k \leq m/2$. Next, assume that $k \geq 3$ is odd; then we find the
recursions 
\[c_{1_k,\sigma_j} = c_{2_{k-2},\sigma_j}+(\zeta^{jk}+\zeta^{-jk})
\qquad \mbox{and}\qquad c_{2_k,\sigma_j}=c_{1_{k-2},\sigma_j}.\]
Finally, the numbers $\breve{n}_w$ have been determined in \cite[1.7.4]{geja}:
\[ \breve{n}_w=\left\{\begin{array}{cl} 1 & \quad \mbox{for $w \in 
\{1_0,1_1,2_1,2_3,1_m\}$},\\
-(-1)^{m/2} & \quad \mbox{for $w=1_{m-1}$},\\
0 & \quad \mbox{otherwise}.\end{array}\right.\]
This allows us, first of all, to verify that Conjecture~\ref{conj42} holds
where
\[ \tilde{\cD}=\{1_0,1_1,2_1,2_3,1_{m-1},1_m\}.\] 
Continuing as in Example~\ref{ctdi}, we conclude that 
Conjecture~\ref{cspec} also holds where 
\[ \cS_L(W)=\{1_W, \mbox{sgn}_1, \mbox{sgn}_2, \mbox{sgn}, \sigma_1\}.\]
For example, for $m=8$ and $b=2$, $a=1$, we obtain for the two left cells 
with $m-2=6$ elements:
\[\fX(W\mid \fC): \quad \begin{array}{cccr} \sigma_1 & 1 & \sqrt{2} & 1 
\\ \sigma_2 & 1 & 0 & -1 \\ \sigma_3 & 1 & -\sqrt{2} & 1\end{array}.\]
\end{exmp}

\begin{exmp} \label{unf4} Let $W$ be of type $F_4$, with generators
and diagram as in Table~\ref{Mcoxgraphs}. Then a weight function $L$ 
is specified by two elements $a,b \in \Gamma_{\geq 0}$ where $a=L(s_0)=
L(s_1)$ and $b=L(s_2)=L(s_3)$. Let us assume that $a>0$ and $b>0$. 
(By the discussion in \cite[\S 2.4]{geja}, the case where $L(s)=0$ for 
some $s \in S$ can always be reduced to the case where all weights are 
strictly positive, possibly by passing to a proper reflection subgroup 
of $W$.) By the symmetry of the diagram, we can also assume that $a\leq b$.
Then, by the results in \cite[\S 4]{mykl}, there are essentially only 
four cases to consider:
\[ a=b, \qquad b=2a, \qquad 2a>b>a, \qquad b>2a.\]
The equal parameter case is already settled by Lusztig \cite{Lusztig79b}.
In the remaining cases it turns out that, for every left cell $\fC$, the 
representation $[\fC]_1$ is multiplicity-free with at most $3$ irreducible 
constituents. Using Algorithm~\ref{algoB} we have checked that
Conjecture~\ref{cspec} holds where the sets $\cS_L(W)$ are given as follows:
\[\begin{array}{l@{\hspace{1mm}}ll}
a=b &:& 1_1, 1_4, 9_1, 9_4, 12, 4_2, 4_5, 8_1, 8_2, 8_3, 8_4; \\
b=2a &:& 1_1, 1_3, 1_4, 2_2, 2_3, 2_4, 4_1, 9_1, 9_2, 9_3, 12, 4_2, 4_3, 
4_4, 4_5, 8_1, 8_2, 8_4; \\
b \not\in\{a,2a\}&:& 1_1, 1_2, 1_3, 1_4, 2_1, 2_2, 2_3, 2_4, 4_1, 9_1,9_2, 
9_3, 9_4, 12, 4_2, 4_3, 4_4, 4_5, 8_1, 8_2, 8_3, 8_4.
\end{array}\]
In all cases where $\fC$ has two irreducible components, the table
$\fX(W \mid \fC)$ is given by:
\[  \begin{array}{ccr} E_1 &1 & 1 \\ E_2 & 1 & -1\end{array} \qquad 
\mbox{where $E_1\in \cS_L(W)$}.\]
We give one particular example where $\breve{n}_d=-1$, for the case
$a=1$, $b=2$: There is a left cell $\fC$ such that $\fC\cap \fC^{-1}=\{d,w\}$
where 
\[ d=s_1s_0s_2s_1 s_0s_2s_1s_2 \quad \mbox{and}\quad w=s_1s_2
s_1s_0s_2s_1s_2s_3s_2s_1s_0 s_2s_1s_2;\]
note that both $l(d)$ and $l(w)$ are even. We have $d \in \tilde{\cD}$, 
$\breve{n}_d=-1$ and 
\[  \fX(W\mid \fC): \quad \begin{array}{ccr} E_1  &1 & 1 \\ E_2 & 1 & -1
\end{array} \qquad \mbox{where} \qquad \begin{array}{l} E_1=4_1\in 
\cS_L(W), \\E_2=16.\end{array}\]
In all cases where $\fC$ has three irreducible components, the table
$\fX(W \mid \fC)$ is given by:
\[  \begin{array}{crrr} E_1 &1 & 2 & 1 \\ E_2 & 1 & -1 & 1 \\
E_3 & 1 & 0 & -1 \end{array} \qquad \mbox{where} \qquad
\begin{array}{l} E_1=12\in \cS_L(W),\\ E_2 \in \{6_1,6_2\},\\  
E_3=16. \end{array}\]
We note the following special behaviour in the case where $b=2a$. By 
\cite[\S 4]{mykl}, there are three left cells $\fC_1,\fC_2,\fC_3$ such 
that
\[ [\fC_1]_1=1_3\oplus 8_3, \qquad [\fC_2]_1=2_1\oplus 9_1, \qquad
[\fC_3]_1=9_1\oplus 8_3.\]
The corresponding representations in $\cS_L(W)$ are $1_3$, $9_1$, $9_1$,
respectively. Since we have $\mbox{Hom}_W([\fC_1]_1,[\fC_3]_1)\neq 0$ and
$\mbox{Hom}_W([\fC_2]_1,[\fC_3]_1)\neq 0$, the three left cells are 
contained in the same two-sided cell. Thus, there are two representations
in $\cS_L(W)$ belonging to this two-sided cell. (This is not an isolated
event: there are many examples in type $B_n$ with unequal parameters as 
well.)~--~This phenomenon can not happen in the equal parameter case 
where every two-sided cell contains a unique special representation (see 
Lusztig \cite[Chap.~5]{LuBook}). 
\end{exmp}

The above examples show that Conjecture~\ref{cspec} holds for $W$ of 
type $I_2(m)$, $F_4$ and any weight function $L$. Thus, the case that 
remains to be dealt with is type $B_n$ with unequal parameters. I have
checked that Conjecture~\ref{cspec} holds for type $B_n$ where $n\in
\{2,3,4,5,6\}$ and any weight function. In general, by the results in 
\cite[\S 22]{Lusztig03}, it is expected that all left cell modules
$[\fC]_1$ in type $B_n$ are multiplicity-free; hence, one may hope that
the tables $\fX(W\mid \fC)$ might be determined as in 
\cite[Prop.~3.11]{Lu4}. If this were true, then Conjecture~\ref{cspec} 
would follow in this case as well.


\affiliationone{M. Geck\\ Institute of Mathematics\\University of 
Aberdeen\\Aberdeen AB24 3UE, UK \email{m.geck@abdn.ac.uk}\\
\url{http://www.abdn.ac.uk/~mth190}}
\end{document}